\documentclass{article}
\usepackage{graphicx} 

\title{Preserving the Ultrapower Axiom by Forcing over Nonoverlapping Extender Models}
\author{Eyal Kaplan\\
Carnegie Mellon University}
\date{\today}
\usepackage[utf8]{inputenc}
\usepackage{amsmath,amssymb,amsthm, xcolor, enumitem, comment, todonotes}
\usepackage{url}
\usepackage[all]{xy}
\usepackage[english]{babel}
\usepackage{hyperref}
\usepackage{tikz-cd}

\theoremstyle{definition}
\newtheorem{definition}{Definition}[section]

\newtheorem{remark}[definition]{Remark}

\theoremstyle{plain}
\newtheorem{theorem}[definition]{Theorem}
\newtheorem{claim}[definition]{Claim}

\newtheorem{lemma}[definition]{Lemma}
\newtheorem{corollary}[definition]{Corollary}

\newcommand{\po}{\mathbb{P}}
\newcommand{\qo}{\mathbb{Q}}
\newcommand{\la}{\langle}
\newcommand{\ra}{\rangle}

\newcommand{\uhr}{\upharpoonright}

\newcommand{\E}{\bar{E}}

\DeclareMathOperator{\dom}{dom}

\DeclareMathOperator{\Ult}{ult}

\begin{document}

\maketitle

\begin{abstract}
    We identify a class of simple yet nontrivial forcing notions that preserve the Ultrapower Axiom (UA) over canonical inner models, reaching the level of a strong cardinal. The forcings we consider are discrete product forcings whose index sets are chosen and spaced according to the structure of the underlying nonoverlapping extender sequence.  As an application, we show that UA is consistent with the existence of a strong cardinal $\kappa$ such that $V\neq \text{HOD}_X$ for every $X\in V_\kappa$, answering a question of Goldberg.
    \end{abstract}

\section{Introduction}

The study of the Ultrapower Axiom ({\rm UA}) has emerged as one of the most prominent areas of research in set theory in recent years. The axiom was introduced and extensively developed by Goldberg, leading to a series of deep and striking consequences for the structure of the set-theoretic universe (see \cite{goldbergtheultrapoweraxiom}). The Ultrapower Axiom asserts that any two ultrapowers of the universe by $\sigma$-complete ultrafilters can be compared by taking further internal ultrapowers, again by $\sigma$-complete ultrafilters. It holds in all currently known canonical inner models and is expected to hold in any canonical inner model constructed in the future.

Establishing consistency results for {\rm UA} by forcing is, in general, a difficult problem. While forcings of rank below the least measurable cardinal preserve {\rm UA}, more substantial forcings may be considerably more destructive. For instance, forcing that adds a new unbounded subset to a measurable cardinal while preserving its measurability may introduce new measures that violate {\rm UA}. Thus, in forcing arguments aimed at preserving {\rm UA}, special care is needed to control the measures that appear in the generic extension.

The motivating question of this paper, raised by Goldberg, is whether a large cardinal $\kappa$ (say, a measurable, strong or  supercompact cardinal) can exist in a model of {\rm UA} satisfying
$V\neq {\rm HOD}_X$
for every $X\in V_\kappa$. The motivation for this question is that it is fairly easy to obtain a model of ${\rm UA}+V\neq{\rm HOD}$ by forcing: one may simply add a Cohen subset of $\omega$ over a canonical inner model. However, in the resulting extension $V={\rm HOD}_X$ for a suitable parameter $X$ of bounded rank. Goldberg's question asks whether one can avoid such trivialities while preserving \({\rm UA}\) by forcing. 

Goldberg's problem is closely connected with the problem of preserving both ${\rm UA}$ and the measurability of $\kappa$ while forcing an unbounded subset of $\kappa$. In \cite{BenNeriaKaplan}, the latter problem was answered in the affirmative by showing that ${\rm UA}$ is consistent with the existence of a measurable cardinal $\kappa$ at which ${\rm GCH}$ fails. The model constructed there, however, satisfies $V={\rm HOD}$. This is due to the incorporation of \textit{self-coding posets} into the forcing construction; these are posets whose generic information is coded by the preservation or destruction of canonically definable stationary sets.\footnote{This idea traces back to the celebrated work of Friedman and Magidor on the possible number of normal measures on a measurable cardinal; see \cite{FriedmanMagidor}.} Additionally, $\kappa$ is not strong in the model of \cite{BenNeriaKaplan}, and, to the best of our knowledge, there had previously been no published forcing construction preserving both ${\rm UA}$ and a strong cardinal.\footnote{Using the same techniques as in \cite{BenNeriaKaplan} to preserve strongness together with ${\rm UA}$ seems plausible, but technically involved.}

In the present paper, we answer Goldberg's question in the affirmative in the context where $\kappa$ is a measurable or a strong cardinal, by isolating a simple class of forcings that preserve {\rm UA} over canonical inner models at these levels. The forcings we consider are discrete product forcings, namely posets of the form
$$\po=\prod_{\alpha\in I}\qo_\alpha,$$
where $I$ is a discrete set of inaccessible cardinals, each $\qo_\alpha$ is $\alpha$-closed, and $\text{rank}(\qo_\alpha)<\min\bigl(I\setminus(\alpha+1)\bigr)$.\footnote{Discrete iterated forcing has appeared recently in the study \rm{UA}: Benhamou and Goldberg \cite{goldbergbenhamou2024applications} study the discrete Magidor iteration and use it to show that the Weak Ultrapower Axiom does not imply the Ultrapower Axiom. However, the present paper uses an entirely different technique, and the role of discreteness is different.}

Discrete product forcings admit a particularly simple presentation, yet enjoy strong fusion properties reminiscent of those of nonstationary-support product forcings. For certain iterated ultrapower embeddings $i\colon V\to M$, these fusion arguments grant total control over the possible generics for $i(\po)$ over $M$. It follows that such embeddings $i$ lift in a unique and canonical way to iterated ultrapowers in the generic extension. This method allows us to show that ${\rm UA}$ is preserved by forcing over Kunen's canonical inner model for a measurable cardinal, $L[U]$ (see Theorem \ref{Thm:ForcingOverLU}). Consequently, this answers Goldberg's question at the level of measurable cardinals (see Corollary \ref{Cor:UA+Meas+VneqHODX}).

To reach the level of a strong cardinal, we work over extender models of the form $L[\vec E]$, where $\vec E$ is a coherent sequence of nonoverlapping extenders, as in Mitchell's minimal inner model for a strong cardinal (see \cite{MitchellInnerModel}). However, a new technical difficulty arises in this setting. The fusion arguments described above allow us to lift certain iterated ultrapower embeddings of the ground model, but it is not clear a priori how to lift \emph{every} normal iteration arising from the $L[\vec E]$-sequence in a canonical way. One of the central ideas of this paper is a technique that allows us to achieve this. The key is to design the forcing itself in accordance with the underlying extender sequence.

Concretely, we require $I$ to be \emph{suitably spaced} with respect to $\vec E$, meaning that whenever $\alpha\in I$ and $\alpha'<\alpha$, $o^{\vec E}(\alpha')<\alpha$, and whenever $\alpha<\kappa$ satisfies $o^{\vec E}(\alpha)>0$, $I\cap\alpha$ is unbounded in $\alpha$. These additional assumptions on the spacing of $I$ allow us to refine the fusion argument so as to control the generic for $i(\po)$ over $M$ for every normal iterate $i\colon L[\vec E]\to M$ arising in the relevant extender comparisons, and hence to lift those comparisons canonically. Consequently, we obtain the following result:

\begin{theorem}\label{Thm:MainThmForStrong}
    Assume that $V = L[\vec{E}]$ is Mitchell's inner model for a strong cardinal, constructed from a coherent sequence of nonoverlapping extenders  $\vec{E}$. Let $\po$ be a suitably spaced discrete product forcing. Let $G\subseteq \po$ be generic over $V$. Then $\kappa$ remains strong in $V[G]$ and $V[G]\vDash{\rm UA}$.
\end{theorem}

As a consequence of Theorem \ref{Thm:MainThmForStrong}, we prove the following.

\begin{corollary}\label{Cor:UA+Strong+FarFromHOD}
    It is consistent with ${\rm UA}$ that there exists a strong cardinal $\kappa$ such that for every $X\in V_{\kappa}$, $V\neq {\rm HOD}_X$.
\end{corollary}

Let us make some final remarks on the role of discreteness in this work. In \cite{FriedmanMagidor,aptercummings2023normalmeasuresonlargecards, BenNeriaKaplan}, the self-coding posets mentioned in a previous paragraph are used to secure the generic at coordinates that cannot be controlled by fusion. The significance of the discreteness assumption is that it allows us to dispense with this self-coding machinery, since the coordinates that would otherwise require coding are no longer active forcing coordinates. This has two advantages. First, it removes the mechanism that made the generic information recoverable in ${\rm HOD}$ in the construction of \cite{BenNeriaKaplan}. Second, the realization of the self-coding construction required an extensive use of the fine structure of the underlying extender model, which is avoided here.\footnote{Fine structure was used substantially, for example, to ensure that the entire forcing construction preserves cardinals. We direct the interested reader to \cite{BenNeriaKaplan} for more details.} Thus, although discreteness restricts the class of forcings under consideration, it provides a remarkably simple alternative to self-coding.

The structure of the paper is as follows. 

\begin{itemize}
    \item In Section \ref{Sec:BackgroundUAandEPA}, we provide a short background on the principles {\rm UA} and its generalization ${\rm EPA}$.
    \item In Section \ref{Sec:DiscreteProductForcing}, we introduce discrete product forcings and establish their basic properties.
    \item In Section \ref{Sec:Measurable}, we prove that discrete product forcings preserve {\rm UA} when forcing over Kunen's model $V=L[U]$; this allows us to illustrate some of our main ideas in the simplest possible setting. 
    \item In Section \ref{Sec:Strong}, we force over Mitchell's canonical inner model for a strong cardinal, which is constructed from a coherent sequence of nonoverlapping extenders. The main results in this section are Theorem \ref{Thm:MainThmForStrong} and Corollary \ref{Cor:UA+Strong+FarFromHOD}.
\end{itemize}

Our notation is standard, with the sole exception of the convention that, in forcing, $p\geq q$ means that $p$ extends $q$, i.e., that $p$ provides more information than $q$. When taking an ultrapower by a $\sigma$-complete ultrafilter $U$, we denote the corresponding elementary embedding by 
$j_U\colon V\to M_U\simeq \Ult(V,U)$.
When we wish to emphasize that $U$ is a measure in a model $N$ and that the ultrapower is taken over $N$, we use the notation
$j_U^N\colon N\to M_U^N\simeq \Ult(N,U)$. We adopt analogous conventions for extender ultrapowers.

\section{{\rm UA} and {\rm EPA}}\label{Sec:BackgroundUAandEPA}

In this section, we survey Goldberg's Ultrapower Axiom and the Extender Power Axiom. 

\begin{itemize}
    \item The \textbf{Ultrapower Axiom} ({\rm UA}) is the assertion that for every pair of $\sigma$-complete ultrafilters $U_0, U_1$ there are $\sigma$-complete ultrafilters $W_0\in \Ult(V, U_1)$, $W_1\in \Ult(V, U_0)$ such that $M^{M_{U_0}}_{W_1} = M^{M_{U_1}}_{W_0}$ and $j^{M_{U_0}}_{W_1}\circ j_{U_0} = j^{M_{U_1}}_{W_0}\circ j_{U_1}$.
    \item The \textbf{Extender Power Axiom} ({\rm EPA}) is the assertion that for every pair of extenders $E_0, E_1$ there are extenders $F_1\in M_{E_0}$, $F_0\in M_{E_1}$ such that $M^{M_{E_0}}_{F_1} = M^{M_{E_1}}_{F_0}$ and $j^{M_{E_0}}_{F_1}\circ j_{E_0} = j^{M_{E_1}}_{F_0} \circ j_{E_1}$. 
\end{itemize}

{\rm UA} holds in all the currently known canonical inner models, and is expected to be consistent with every large cardinal axiom.

{\rm EPA} holds in the minimal canonical inner model with a strong cardinal. However, unlike {\rm UA}, {\rm EPA} is known to be inconsistent with stronger assumptions:

\begin{theorem}[Goldberg, { \cite[Theorem 2.3]{GoldbergGeneralizationsUA}}]\label{Thm:GoldbergExtenderAxiomCanBeInconsistent}
    Assume that $\kappa$ is $\kappa$-compact.\footnote{That is, every $\kappa$-complete filter on $\kappa$ can be extended to a $\kappa$-complete ultrafilter.} Then {\rm EPA} doesn't hold.
\end{theorem}

\begin{lemma}[Goldberg]\label{Lem:UAFollowsfromEPA}
    ${\rm EPA}$ implies ${\rm UA}$.
\end{lemma}

\begin{proof}
    Fix $\sigma$-complete ultrafilters $U_0, U_1$. By ${\rm EPA}$, there are extenders $E_1\in M_{U_0}$, $E_0\in M_{U_1}$ such that $M:=M^{M_{U_1}}_{E_0} = M^{M_{U_0}}_{E_1}$ and $j^*:=j^{M_{U_0}}_{E_1}\circ j_{U_0} = j^{M_{U_1}}_{E_0}\circ j_{U_1}$. Denote $a_0 = j^{M_{U_0}}_{E_1}([id]_{U_0})$ and $a_1 = j^{M_{U_1}}_{E_0}([id]_{U_1})$. Let $N$ be the transitive collapse of 
    $$H =\mathcal{H}^M( j^*[V]\cup \{a_0,a_1\} ), $$
    where $\mathcal{H}^M$ denotes the Skolem Hull inside $M$ (with respect to fixed in advance Skolem functions). Let $\pi\colon H\to N$ be the transitive collapse map. Since $j^{M_{U_1}}_{E_0}, j^{M_{U_0}}_{E_1}$ have their ranges contained in $H$, we can consider the embedding $i_0 = \pi\circ j^{M_{U_0}}_{E_1}\colon M_{U_0}\to N$ and $i_1 = \pi\circ j^{M_{U_1}}_{E_0}\colon M_{U_1}\to N$, and note that $i_0\circ j_{U_0} = i_1\circ j_{U_1}$. It thus remains to show that $i_0,i_1$ are internal ultrapower embeddings.

    In order to prove that $i_0$ is an internal ultrapower embedding, it suffices to prove that $N = \mathcal{H}^{N}\left( i_0[M_{U_0}]\cup \{ \pi(a_1) \} \right)$. Indeed, once we prove this, let $\lambda_1$ be the least cardinal in $M_{U_0}$ such that $ \pi(a_1)<i_0(\lambda_1)$, and define $W_1 = \{ X\subseteq \lambda_1 \colon \pi(a_1)\in i_0(X) \}$. Then $W_1\in M_{U_0}$ (since $i_0$ is definable in $M_{U_0}$), and $N\simeq \Ult( M_{U_0}, W_1 )$. Under this identification, $i_0( f )( \pi(a_1) ) = j_{W_1}(f)(  \pi(a_1))$ for every function $f\colon \lambda_1\to M_{U_0}$ in $M_{U_0}$. In particular, $i_0 = j_{W_1}$.

    Thus, it remains to prove that $N = \mathcal{H}^{N}\left( i_0[M_{U_0}]\cup \{ \pi(a_1) \} \right)$. The inclusion $(\supseteq)$ is easy, so let us concentrate on the other one. Fix $x\in N$. Since $N \simeq \mathcal{H}^M\left( j^*[V]\cup \{a_0,a_1\} \right)$, $x$ has the form $\pi( j^*(g)(a_0,a_1) )$ for some function $g\in V$. Namely 
    $$ x = i_0(  j_{U_0}(g)  ) \bigl( i_0\left( [ id ]_{U_0}\right), \pi(a_1) \bigr). $$
    Let $h\in V$ be a function whose image consists of functions, such that for every $a,b$, $h(a)(b) = g(a,b)$. It thus follows that
    $$ x = i_0\left( j_{U_0}(h)( [id]_{U_0} ) \right)( \pi(a_1) ),$$
    Since $j_{U_0}(h)( [id]_{U_0} )\in M_{U_0}$, we deduce that indeed $x\in \mathcal{H}^N( i_0[ M_{U_0} ]\cup \{ \pi(a_1) \} )$.

    The proof that $i_1$ is an internal ultrapower embedding is symmetric.
\end{proof}

\begin{definition}
    Let $M,N$ be inner models. An embedding $i\colon M\to N$ is called \textit{close} if for every $x\in N$, $i^{-1}[x]\in M$.
\end{definition}

The proof of Lemma \ref{Lem:UAFollowsfromEPA} really used only the fact that $j^{M_{U_0}}_{E_1}$ and $j^{M_{U_1}}_{E_0}$ are close embeddings (which follows from the fact that both are definable extender embeddings over their domains). We can generalize the proof to show:

\begin{lemma}[Goldberg]\label{Lem:ComparisonByCloseEmbeddings}
    Assume $U_0, U_1$ are $\sigma$-complete ultrafilters. Let $N$ be an inner model of $V$ such that there are close embeddings $j_0\colon M_{U_1}\to N$ and $j_1\colon M_{U_0}\to N$ with $j_0\circ j_{U_1} = j_1\circ j_{U_0}$. Then there are $\sigma$-complete ultrafilters $W_0\in M_{U_1}, W_1\in M_{U_0}$ such that $\Ult(M_{U_0}, W_1)\simeq \Ult(M_{U_1}, W_0)$ and $j^{M_{U_1}}_{W_0}\circ j_{U_1} = j^{M_{U_0}}_{W_1}\circ j_{U_0}$. 
\end{lemma}

\begin{proof}
    Mimic the proof of Lemma \ref{Lem:UAFollowsfromEPA} by replacing the extender ultrapowers $j^{M_{U_1}}_{E_0}, j^{M_{U_0}}_{E_1}$ with $j_0, j_1$, respectively. The embeddings $i_0\colon M_{U_0}\to N, i_1\colon M_{U_1}\to N$ are close since $j_1, j_0$ are close. In particular, the measures 
    $$W_k = \{ X\subseteq \lambda_k \colon \pi(a_k)\in i_{1-k}(X)  \} = i_{1-k}^{-1}[ \{ Y\subseteq i_{1-k}(\lambda_k)\colon \pi(a_k)\in Y \} ]$$
    belong to $M_{U_{1-k}}$ for every $k\in \{0,1\}$. Thus, $M_{U_0}, M_{U_1}$ can be compared by internal ultrapowers.
\end{proof}

\section{Discrete product forcings}\label{Sec:DiscreteProductForcing}

\begin{definition}\label{Def:discrete-product-forcings}
Fix a cardinal $\kappa$. A \emph{discrete product forcing} of length $\kappa$ is a product forcing
$$\po=\prod_{\alpha\in I}\qo_\alpha,$$
where:
\begin{enumerate}
    \item $I\subseteq\kappa$ is a discrete set of inaccessible cardinals; that is, for every $\alpha\in I$, the set $I\cap\alpha$ is bounded in $\alpha$.
    
    \item For every $\alpha\in I$, $\langle\qo_\alpha,\leq_{\qo_\alpha}\rangle$ is an $\alpha$-closed partial order satisfying
    \[
    \text{rank}(\qo_\alpha)<\min(I\setminus(\alpha+1)).
    \]
    If $\min(I\setminus (\alpha+1))$ doesn't exist, require instead $\text{rank}(\qo_\alpha)<\kappa$.
\end{enumerate}
Thus, a condition $p\in\po$ is a function $
p:I\to V_\kappa$
 such that $p(\alpha)\in\qo_\alpha$ for every $\alpha\in I$. The ordering on $\po$ is the coordinatewise ordering: for $p,q\in\po$, $
p\leq q$ (i.e., $q$ extends $p$) if and only if $p(\alpha)\leq_{\qo_\alpha} q(\alpha)
$ for every $\alpha\in I$.
\end{definition}

\begin{remark}
For every $\gamma<\kappa$, we may factor $\po$ at $\gamma$, as usual, into its initial and tail parts. Let
$$\po_\gamma=\prod_{\alpha\in I\cap\gamma}\qo_\alpha$$
and
$$\po\setminus\gamma=\prod_{\alpha\in I\setminus\gamma}\qo_\alpha.$$
Then the map
$p\longmapsto \langle p\restriction(I\cap\gamma),p\restriction(I\setminus\gamma)\rangle$
is a canonical isomorphism between $\po$ and $\po_\gamma\times(\po\setminus\gamma)$. Accordingly, if $G\subseteq\po$ is generic over $V$, we write $G_\gamma=G\restriction\po_\gamma$ for the induced $\po_\gamma$-generic filter. The remaining part of the extension is then obtained by forcing over $V[G_\gamma]$ with the tail forcing $\po\setminus\gamma$, so that
$V[G]=V[G_\gamma][G\restriction(\po\setminus\gamma)].$

Since $\po$ is a product rather than an iteration, the tail forcing $\po\setminus \gamma$ is already an element of the ground model and is unchanged when passing to the intermediate extension $V[G_\gamma]$. Moreover, the hypotheses of Definition \ref{Def:discrete-product-forcings} ensure that the forcing $\po\setminus \gamma$ is $\gamma$-closed.
\end{remark}

\begin{lemma}{\rm{(Discrete Fusion Lemma)}}\label{Lem:DiscreteFusion}
    Let $\kappa$ be a cardinal, $I\subseteq \kappa$ a set of inaccessible cardinals, and let $\po = \prod_{\alpha\in I}\qo_\alpha$ be a discrete product forcing satisfying the hypotheses of Definition \ref{Def:discrete-product-forcings}. Assume that $\la d(\alpha) \colon \alpha<\kappa \ra$ is a sequence of dense open subsets of $\po$. Then for every $p\in \po$ there is $p^*\in \po$ extending $p$ such that for every $\alpha<\kappa$ and $\beta\leq\sup(I\cap \alpha)$, the set 
    $$ \{ r\in \po_{\alpha} \colon r{}^\frown p^*\setminus \alpha\in d(\beta) \} $$
    is dense open above $p^*\uhr\alpha$ in $\po_{\alpha}$.
\end{lemma}

\begin{proof}
    We construct a sequence of conditions $\la p_i \colon i
    \in I\cup \{\kappa \}\ra$, all of them extend $p$, such that the following properties hold:
    \begin{enumerate}
        \item \label{Item:Fusioncoherence}For every $i\in I\cup\{\kappa\}$ and $j\in I\cap i$, $p_j \leq p_i$ and $p_i\uhr (j+1) = p_j\uhr (j+1)$.
        \item \label{Item:DenseSetCapturing}For every $i\in I\cup \{\kappa\}$ and $\beta<\kappa$ such that
        $\beta\leq\sup(I\cap i)$,\footnote{It is possible that $\sup(I\cap i) = \kappa$ (e.g., $i=\kappa$ and $I$ is unbounded in $\kappa$), so the requirement $\beta<\kappa$ is needed.}  the set $\{r\in \po_{i} \colon r{}^\frown p_i\setminus i\in d(\beta)
        )  \}$ is a dense subset of $\po_i$ above $p_i\uhr i$.
    \end{enumerate}

    Suppose that $i\in I\cup 
    \{\kappa\}$ and the sequence $\la p_j \colon j\in I\cap i \ra$ has been constructed. We first construct a condition $q$ which is an upper bound of $\langle p_j \colon j<i\rangle$. Let $i^*\leq i$ be the least ordinal strictly greater than every element of $I\cap i$ (namely, if $\max(I\cap i)$ exists, $i^*$ is its successor; otherwise, $i^* = \sup(I\cap i)$). Note that $\po_i =\po_{i^*}$. Let $q\in \po$ be a condition such that
    \begin{itemize}
        \item $q\uhr i = q\uhr i^* = \bigcup_{j\in I\cap i} p_j\uhr (j+1)$.
        \item $q\setminus i$ an upper bound of $\la p\setminus i \ra{}^\frown\la p_j\setminus i \colon j\in I\cap i \ra$.
    \end{itemize}

    Let us elaborate on why such a condition $q$ exists. The first bullet point follows from the coherence requirement in clause \ref{Item:Fusioncoherence}. Moreover, the resulting initial segment extends $p\uhr i$, since every $p_j$ extends $p$. 

    For the second bullet point, if $i=\kappa$, then $\po\setminus i$ is trivial. Otherwise, $i\in I$, and the discreteness of $I$ implies that
    $\sup(I\cap i)<i$. In this case, $\po\setminus i$ is $i$-closed, while
    $|I\cap i|<i$. Hence the sequence of tails
    $$ \la p\setminus i\ra{}^\frown \la p_j\setminus i\colon j\in I\cap i\ra$$
    has an upper bound in $\po\setminus i$. This proves the existence of $q$.
    
    Next, we construct the condition $p_i \geq q$. Denote $\gamma^*:= |\po_{i}|$ and fix an enumeration $\la r_\gamma\colon \gamma<\gamma^* \ra$ of $\po_{i}$. Construct\footnote{The construction is vacuous in the case where $\po\setminus i$ is trivial.} an increasing sequence of conditions $\la s_
    \gamma\colon \gamma\leq\gamma^* \ra \in \po\setminus i$ such that
    \begin{itemize}
        \item $s_0 = q\setminus i$.
        \item for every limit $\gamma\leq \gamma^*$, $s_\gamma$ is an upper bound of $\la s_{\gamma'}\colon \gamma'<\gamma \ra$.
        \item for every $\gamma<\gamma^*$, $s_{\gamma+1}\geq s_\gamma$ satisfies that for every $\beta\leq\sup(I\cap i)$ below $\kappa$, there is some $r'_{\gamma,\beta}\geq r_\gamma$ such that $r'_{\gamma,\beta}{}^\frown s_{\gamma+1} \in d(\beta)$.
    \end{itemize}
    
    The second and third bullets above can be achieved by noting that $\po\setminus i$ is more than $|\po_i|\cdot|\sup(I\cap i)+1|$-closed. Indeed, as above, this is clear for $i=\kappa$, and if $i\in I$, then $I\cap i$ is bounded in $i$, $|\po_i|<i$, and $\po\setminus i$ is $i$-closed. For the third bullet, for a fixed $\gamma<\gamma^*$, recursively go through the relevant values of $\beta$, using the density of $d(\beta)$ to strengthen the tail $s_{\gamma+1}\geq s_{\gamma}$ so that a condition $r'_{\gamma,\beta}$ which extends $r_\gamma$ and satisfies 
    $$ r'_{\gamma, \beta}{}^\frown s_{\gamma+1}\in d(\beta). $$
    exists. At limit stages take upper bounds in $\po\setminus i$  (namely, $s_\gamma$ for $\gamma$ limit is an upper bound of $\la s_\gamma' \colon  \gamma'<\gamma\ra$). The openness of the sets $d(\beta)$ ensures that requirements obtained at earlier stages are preserved.

    Finally, define $p_i = q\uhr i {}^\frown s_{\gamma^*}$. By the above construction, $p_i$ satisfies that for every $\beta\leq \sup(I\cap i)$ with $\beta<\kappa$, the set $\{r\in \po_{i} \colon r{}^\frown p_i\setminus i\in d(\beta))  \}$ is a dense subset of $\po_i$ above $p_i\uhr i$. Indeed, given $r\in \po_i$, let $\gamma<\gamma^*$ be such that $r = r_\gamma$. Then $r'_{\gamma,\beta}\geq r$ satisfies $r'_{\gamma,\beta}{}^\frown p_i\setminus i\in d(\beta)$.

    This concludes the inductive construction. Let $p^* = p_\kappa$. Then $p^*\in \po$ is an upper bound of $\la p_i \colon i\in I \ra$. It remains to show that $p^*$ is as desired. Fix $\alpha<\kappa$. We argue that for every $\beta\leq\sup(I\cap \alpha)$,
    $$ \{ r\in \po_{\alpha} \colon  r{}^\frown p^*\setminus \alpha\in d(\beta) \}$$
    is dense open above $p^*\uhr \alpha$. Let $i = \min((I\cup \{\kappa\})\setminus \alpha)\geq \alpha$. Since $p^*\geq p_i$ and each $d(\beta)$, $\beta<\kappa$, is open, clause \ref{Item:DenseSetCapturing} implies that for every $\beta<\kappa$ with $\beta\leq\sup(I\cap i)$, the set
    $$ \{ r\in \po_{i} \colon  r{}^\frown p^*\setminus i\in d(\beta) \}$$
    is dense above $p^*\uhr i$. By the choice of $i$, $\po_i = \po_\alpha$, and thus for every $\beta\leq\sup(I\cap i)$ below $\kappa$,
    $$ \{ r\in \po_{\alpha} \colon  r{}^\frown p^*\setminus \alpha\in d(\beta) \} $$
    is dense above $p^*\uhr \alpha$, as desired.
\end{proof}

\begin{lemma}\label{Lem:PreservingCardinals}
    Assume {\rm GCH} and let $\kappa$ be an inaccessible cardinal. Let $\po = \prod_{\alpha\in I}\qo_\alpha$ be as in definition \ref{Def:discrete-product-forcings}. Assume that for every $\alpha\in I$, $\qo_\alpha$ preserves cardinals. Then $\po$ preserves cardinals.
\end{lemma}

\begin{proof}
    Assume by contradiction that $\po$ collapses cardinals, and let $\lambda$ be the least cardinal collapsed. By minimality, $\lambda$ is a successor cardinal in $V$. By {\rm GCH}, $|\po| \leq \kappa^{+}$, and thus $\po$ is $\kappa^{++}-c.c.$. This in particular implies that $\lambda\leq \kappa^+$. 
    
    Let us argue that $\lambda\leq \kappa$, namely $\kappa^+$ is preserved. This is a routine application of the Discrete Fusion Lemma (Lemma \ref{Lem:DiscreteFusion}). We can assume first that $I$ is unbounded in $\kappa$ (else, $|\po|<\kappa$ and clearly $\po$ preserves $\kappa^{+}$). Let $\dot{f}$ be a $\po$-name, and $p\in \po$ forces that $\dot{f}\colon \kappa\to \kappa^+$ is an increasing function. It suffices to prove that there is $q\geq p$ and $\gamma<\kappa^+$ such that $q\Vdash \text{Im}(\dot{f})\subseteq \check{\gamma}$. To see this, define, for every $\alpha<\kappa$, the set $d(\alpha)\subseteq \po$ consisting of conditions that decide the value of $\dot{f}(\alpha)$. For every $\alpha<\kappa$, let $\beta_\alpha = \sup(\alpha\cap I)$. By the Discrete Fusion Lemma, there exists $q\geq p$ such that for every $\alpha<\kappa$,  
    $$ \{  r\in \po_\alpha \colon r{}^\frown q\setminus \alpha\in d(\beta_\alpha) \} $$
    is a dense open subset of $\po_\alpha$ above $q\uhr \alpha$. Define, for each $\alpha<\kappa$,
    $$ \gamma_\alpha = \sup\{  \delta \colon \exists r\in \po_{\alpha} \   \left( r{}^\frown q\setminus \alpha \Vdash \dot{f}(\check{\beta}_\alpha) = \check{\delta} \right) \}. $$
    It follows that $q\Vdash \dot{f}(\check{\beta}_\alpha)\leq \check{\gamma}_\alpha$. Let $\gamma = (\sup \{\gamma_\alpha\colon \alpha<\kappa\})+1<\kappa^+$. Since $\dot{f}$ is forced to be an increasing function and $I$ is unbounded in $\kappa$, we deduce that $q\Vdash \mbox{Im}(\dot{f})\subseteq \check{\gamma}$, as desired. 

    This implies that $\lambda\leq \kappa$. Next, let us argue that $\po_\lambda$ must have collapsed $\lambda$. Denote $\mu = \min(I\setminus \lambda)$ (if no such $\mu$ exists, $\po\setminus \lambda$ is a trivial forcing). Factor $\po = \po_\lambda \times \po\setminus \lambda$. Note that $\lambda\notin I$ since $\lambda$ is a successor cardinal. In particular, $\po_\lambda \times \po\setminus \lambda$ is a product of a $\mu$-c.c. forcing with a $\mu$-closed forcing. By Easton's Lemma, in $V[G_{\lambda}]$,  $\po\setminus \lambda$ is $\mu$-distributive, and cannot collapse $\lambda$. 

    It thus follows that $\po_{\lambda}$ collapsed $\lambda$. Write $\lambda=\nu^+$. 
    
    If $\nu\in I$ then $I\cap \nu$ is bounded in $\nu$, by the discreteness of $I$. By the hypotheses of the lemma, $\qo_\nu$ preserves cardinals. Moreover, $\Vdash_{\qo_\nu} |\check{\po}_\nu|<\check{\nu}$. It thus follows that $\po_\lambda =  \qo_\nu \times \po_{\nu}$ preserves $\lambda$.

    Thus, assume that $\nu\notin I$. In particular, $\po_\lambda = \po_\nu$. 
    
    If $I$ is unbounded in $\nu$, the same discrete fusion argument given above shows that $\po_\nu$ cannot collapse $\lambda= \nu^+$.

    Thus, assume that $I\cap \nu$ is bounded in $\nu$. Denote $\delta  =\sup(I\cap \nu)<\nu$.
    
     If $\delta\in I$ (namely, $\delta = \max(I\cap \nu)$), then $\po_\lambda = \qo_\delta\times \po_\delta$, and since $\qo_\delta$ preserves cardinals and $\Vdash_{\qo_\delta} |\check{\po}_\delta|<\check{\delta}$ (by the discreteness of $I$), $\po_\lambda$ preserves $\lambda$. 
    
    If $\delta\notin I$, then $\po_\lambda=\po_\delta$. Since $\delta$ is a limit point of $I$, every iterand occurring in $\po_\delta$  has cardinality less than $\delta$. Hence, by {\rm GCH}, $|\po_\lambda|\leq 2^\delta\leq\nu$. Thus $\po_\lambda$ is $\nu^+$-c.c., and therefore preserves
    $\lambda=\nu^+$.
\end{proof}

We conclude this section by proving that discrete product forcings cannot be used to increase $2^\kappa$.

\begin{lemma}\label{Lem:PreservingGCH}
    Assume {\rm GCH}, let $\kappa$ be an inaccessible cardinal,  and let $\po = \prod_{\alpha\in I}\qo_\alpha$ be a discrete product forcing satisfying the hypotheses of Definition \ref{Def:discrete-product-forcings}. Then $\Vdash_{\po} 2^{\kappa} = \kappa^+ $.\footnote{The part of Lemma \ref{Lem:PreservingCardinals} showing that $\kappa^+$ is preserved remains valid (even when the forcings $\qo_\alpha$ collapse cardinals). Thus, the cardinal $\kappa^+$ appearing here is the same as $\kappa^+$ computed in $V$.}
\end{lemma}

\begin{proof}
    This is another classical application of the Fusion Lemma,  adapted to the discrete setting. As usual, the nontrivial case is when $I$ is unbounded in $\kappa$. Assume that $\dot{X}$ is a $\po$-name for a subset of $\kappa$, and this is forced by some $p\in \po$. By the Discrete Fusion Lemma (Lemma \ref{Lem:DiscreteFusion}) there exists $q\geq p$ such that for every $\alpha<\kappa$, letting $\beta_\alpha = \sup(I\cap \alpha)$, the set
    $$ E_\alpha:=\{ r\in \po_\alpha \colon r\Vdash \exists X_\alpha\subseteq \beta_\alpha \ \left(q\setminus \alpha \Vdash \dot{X}\cap \beta_\alpha = X_\alpha \right) \}  $$
    is dense open in $\po_\alpha$ above $q\uhr \alpha$. Fix a maximal antichain $\mathcal{A}_\alpha\subseteq E_\alpha$ above $q\uhr \alpha$ so that every $r\in \mathcal{A}_\alpha$ has an associated $\po_\alpha$-name $\dot{X}_\alpha(r)$, for which 
    $$ r{}^\frown (q\setminus \alpha)\Vdash \dot{X}\cap \check{\beta}_\alpha = \dot{X}_\alpha(r). $$
    Let $\dot{X}^*_\alpha$ be the mix-name of the sequence $\la \dot{X}_\alpha(r) \colon r\in \mathcal{A}_\alpha  \ra$ (namely, $\dot{X}^*_\alpha$ is a $\po_\alpha$-name, such that for every $r\in \mathcal{A}_\alpha$, $r\Vdash \dot{X}^*_\alpha = \dot{X}_\alpha(r)$). It thus follows that $$q\Vdash \dot{X} = \bigcup_{\alpha<\kappa} \left( \dot{X}^*_\alpha \right)_{\dot{G}_{\alpha}}.$$

    Overall, we proved that for every $X\in \left( \mathcal{P}(\kappa) \right)^{V[G]}$ there is a sequence of names $\vec{\dot{X}}^* = \la \dot{X}^*_\alpha \colon \alpha<\kappa \ra\in {}^\kappa (V_\kappa)$ such that 
    $$X = \bigcup_{\alpha<\kappa} \left( \dot{X}^*_\alpha \right)_{G_\alpha}.$$ By {\rm GCH}, there are at most $\kappa^+$-many such sequences $\vec{\dot{X}}^*$. It thus follows that $\left| \left( \mathcal{P}(\kappa) \right)^{V[G]} \right| = \kappa^+$, as desired.
\end{proof}

\begin{remark}
Lemma \ref{Lem:PreservingGCH} demonstrates that forcing {\rm UA} together with the failure of {\rm GCH} at a measurable cardinal requires the use of non-discrete forcings. A forcing construction that achieves this appears in \cite{BenNeriaKaplan}. The proof  relies heavily on a combination of various fusion arguments and  fine-structure based forcing techniques. We refer the interested reader to \cite{BenNeriaKaplan} for further details.
\end{remark}

\section{A measurable cardinal}\label{Sec:Measurable}

Our main result in this section is the following theorem.

\begin{theorem}\label{Thm:ForcingOverLU}
    Assume that $V= L[U]$ and $\kappa$ is the unique measurable cardinal. Fix a discrete set of inaccessibles $I\subseteq \kappa$ and let $\po = \prod_{\alpha\in I}\qo_\alpha$ be a discrete product forcing, satisfying the hypotheses of Definition \ref{Def:discrete-product-forcings}. Let $G\subseteq \po$ be generic over $V$. Then in $V[G]$, $\kappa$ remains measurable and the Ultrapower Axiom holds.
\end{theorem}

It is a well known fact due to Kunen (see \cite{KunenMeasures}) that every $\sigma$-complete ultrafilter in $L[U]$ is Rudin-Keisler equivalent to $U^n$ for some $n<\omega$. Let us analyze how such measures lift when forcing with a discrete product  forcing. The following proof is a simple modification of \cite[Theorem 1.4]{Kaplan}.

\begin{lemma}\label{Lem:LiftsOfPowersOfU}
    Assume that $\kappa$ is a measurable cardinal and $U\in V$ is a normal measure on $\kappa$. Let $I\subseteq \kappa$ be a discrete set of inaccessible cardinals, and assume that $\po = \prod_{\alpha\in I}\qo_\alpha$ is a discrete product forcing satisfying the hypotheses of Definition \ref{Def:discrete-product-forcings}. Then for every $1\leq n<\omega$, $U^n$ generates a measure on $[\kappa]^n$ in $V[G]$.
\end{lemma}

    \begin{proof}
        We assume throughout the proof that $I$ is an unbounded subset of $\kappa$. The proof for the case in which $I$ is bounded follows from the Levy-Solovay Theorem (see \cite{levysolovay1967}).
        
        Let $H\subseteq j_{U^n}(\po)$ be the filter generated from $j_{U^n}[G]$, that is,
        $$ H =  \{ q\in j_{U^n}(\po)  \colon \exists p\in G, \  q\leq j_{U^n}(p) \}.$$
        We argue that $H$ is $j_{U^n}(\po)$-generic over $M_{U^n}$. Clearly $H$ is a filter on $j_{U^n}(\po)$, so it suffices to prove that $H$ meets all the dense open subset of $j_{U^n}(\po)$ which belong to $M_{U^n}$.

        Fix a dense open subset $D\in M_{U^n}$ of $j_{U^n}(\po)$. For every increasing sequence $\vec{\nu} = \la \nu_0,\ldots, \nu_{n-1} \ra\in [\kappa]^n$, let $d(\vec{\nu})\subseteq \po$ be a dense open subset such that $D = j_{U^n}(\vec{\nu}\mapsto d(\vec{\nu}))(\kappa_0,\ldots, \kappa_{n-1})$. For every $1\leq k\leq n-1$ and $\la \nu_0,\ldots, \nu_{k-1} \ra\in [\kappa]^{k}$, let 
        \begin{align*}
            d_k(\nu_0,\ldots, \nu_{k-1}) = \{r\in \po  \colon & \exists A\in U^{n-k} \ \  \forall \la \nu_k,\ldots, \nu_{n-1} \ra\in A\setminus (\nu_{k-1}+1), \\
            & r\in d(\nu_0,\ldots, \nu_{k-1}, \nu_k, \ldots, \nu_{n-1} )  \}.
        \end{align*}
        For sake of completeness, we provide the definitions for $k=0$ and $k=n$. For $k=0$, let
        \begin{align*}
            d_0 = \{r\in \po  \colon & \exists A\in U^{n} \ \  \forall \la \nu_0,\ldots, \nu_{n-1} \ra\in A, \ r\in d(\nu_0,\ldots, \nu_{n-1} )  \}.
        \end{align*}
        For $k=n$ and for every $\vec{\nu} = \la \nu_0,\ldots, \nu_{n-1} \ra\in [\kappa]^n$, let $d_{n}(\vec{\nu}) := d(\vec{\nu})$.     
        
        By our assumption, each of the sets $d_n(\vec{\nu})$ is dense open. We argue that $d_0$ is a dense open subset of $\po$. This follows from the following claim.
        
    \begin{claim}\label{Claim:DownwardsInductionMultivarFusion}
        For every $0\leq k\leq n$ and $\la \nu_0,\ldots, \nu_{k-1} \ra\in [\kappa]^{k}$, the set $d_k(\nu_0,\ldots, \nu_{k-1})$ is a dense open subset of $\po$. 

    \end{claim}
    \begin{proof}
        We prove the claim by an inverse induction on $k$. Namely,  assume that the statement is true for some $1\leq k\leq n$, and prove the statement for $k-1$ (the induction basis, which is the case $k=n$, is known to hold by our definition of the sets $d_n(\vec{\nu})$). 

        Fix $\la \nu_0,\ldots, \nu_{k-2}\ra\in [\kappa]^{k-1} $, and let us argue that the set $d_{k-1}(\nu_0,\ldots, \nu_{k-2})$ is dense and open. Openness follows from the fact that the sets $d(\vec{\nu})$ are themselves open. Thus, we concentrate on density. 
        
        Fix $p\in \po$. By the induction hypothesis, for every $\nu\in (\nu_{k-2}, \kappa)$, the set
        \begin{align*}
            e_{\nu_0,\ldots, \nu_{k-2}}(\nu):= d_{k}(\nu_0,\ldots, \nu_{k-2}, \nu)
        \end{align*}
        is known to be a dense open subset of $\po$. Thus, by the Discrete Fusion Lemma (Lemma \ref{Lem:DiscreteFusion}), there exists $p'\in \po$ extending $p$ such that for every $\alpha<\kappa$ and for every $\nu \leq \sup(I\cap \alpha)$,
        $$ \{ r\in \po_\alpha \colon r{}^\frown p'\setminus \alpha\in e_{\nu_0,\ldots, \nu_{k-2}}(\nu) \} $$
        is dense open above $p'\uhr \alpha$ in $\po_\alpha$. Next, let $C = \{ \nu<\kappa \colon \sup(I\cap \nu) = \nu \} $, and note that $C$ is a club in $\kappa$ (by the unboundedness of $I$) which is disjoint from $I$ (by the discreteness of $I$). Fixing $\nu\in C\setminus (\nu_{k-2}+1)$, let $\alpha(\nu):= \min(I\setminus (\nu+1))$. Note that $\alpha(\nu)<\kappa$ since $I$ is unbounded in $\kappa$. Also, since $\nu\notin I$, $\po_\nu = \po_{\alpha(\nu)}$. It follows that for each such $\nu$ there exists a condition $r(\nu)\in \po_{\alpha(\nu)} = \po_{\nu}$ extending $p'\uhr \nu$,  for which 
        $$ r(\nu){}^\frown p'\setminus \nu\in d_{k}{(\nu_0,\ldots, \nu_{k-2},\nu)}. $$

        Since $C$ is a club in $\kappa$, $C\setminus (\kappa_{n-2}+1)\in U$. Hence, we can consider the condition $p^* = [\nu\mapsto r(\nu)]_U$. By the normality of $U$, $p^*\in (j_U(\po))_\kappa = \po$. Also, $p^*\geq p'$ since for each $\nu$, $r(\nu)\geq p'\uhr \nu$. Denote $A:=\{ \nu<\kappa \colon p^*\uhr \nu = r(\nu) \}\in U$. In particular, for every $\nu\in A$,
        $$ (p^*\uhr \nu)^{\frown} p'\setminus \nu \in d_{k}(\nu_0,\ldots, \nu_{k-2}, \nu).$$
        Since the sets $d_{k}(\nu_0,\ldots, \nu_{k-2}, \nu)$ are open, we deduce that for every $\nu\in A$, $p^*\in d_{k}(\nu_0,\ldots, \nu_{k-2}, \nu)$. In particular, there exists a set $A(\nu)\in U^{n-k}$ such that for every $\la \nu_{k}, \ldots, \nu_{n-1} \ra\in A(\nu)$,
        $$ p^*\in d(\nu_0,\ldots, \nu_{k-2}, \nu, \nu_k, \ldots, \nu_{n-1} ). $$
        Finally, let
        $$ A^* = \{  \la \nu, \nu_{k}, \ldots, \nu_{n-1} \ra\in [\kappa]^{n-k+1} \colon  \nu\in A \mbox{ and } \la \nu_{k},\ldots, \nu_{n-1} \ra\in A(\nu)\}. $$
        Since for every $\nu\in A$, $A(\nu)\in U^{n-k}$, we deduce that $A^*\in U^{n-k+1}$. This concludes the proof of Claim \ref{Claim:DownwardsInductionMultivarFusion}.
    \end{proof}
    It thus follows that $d_0\subseteq \po$ is dense open. In particular, for some $p\in G$, there exists $A\in U^{n}$ such that for every $\vec{\nu} = \la \nu_0,\ldots,\nu_{n-1}  \ra\in A$, $p\in d(\vec{\nu})$. In particular, $j_{U^n}(p)\in D$, as desired.

    This concludes the proof that $H$ is $j_{U^n}(\po)$-generic over $M_{U^n}$. 
    
    Define a measure $W\in V[G]$ on $[\kappa]^n$, consisting of all the interpretations of $\po$-names $\dot{X}$, for which there exists $p\in G$ such that $p\Vdash \dot{X}\subseteq [\check{\kappa}]^n$, and 
    $$j_{U^n}(p)\Vdash \la \check{\kappa}_0,\ldots, \check{\kappa}_{n-1} \ra \in j_{U^n}(\dot{X}).$$
    It remains to argue that $W$ is a $\kappa$-complete ultrafilter, and $W$ is generated in $V[G]$ from $U^n$.

    It is routine to prove that $W$ extends $U^n$ and is closed upwards and under finite intersections. To prove that $W$ is a $\kappa$-complete ultrafilter on $[\kappa]^n$ in $V[G]$, fix an ordinal $\eta<\kappa$ and a partition $\la X_i \colon i<\eta \ra$ of $[\kappa]^n$. We argue that for some $i<\eta$, $X_i\in W$. Fix a condition $p\in G$ which forces that $\la \dot{X}_i \colon i<\check{\eta} \ra$ is a partition of $[\kappa]^n$. It follows that $j_{U^n}(p)$ forces that $\la j_{U^n}(\dot{X}_i) \colon i<\check{\eta} \ra$ is a partition of $[\kappa_n]^n$. Define
    $$D = \{ q\in j_{U^n}(\po) \colon \exists i<\eta, \  \ q\Vdash  \la \check{\kappa}_0,\ldots, \check{\kappa}_{n-1} \ra \in j_{U^n}(\dot{X}_i) \}.$$
    It's straightforward that $D$ is a dense open subset of $j_{U^n}(\po)$ above the condition $j_{U^n}(p)\in H$. In particular, $D$ intersects $H$, and, since $D$ is open and $H$ is generated from $j_{U^n}[G]$, we can assume that for some $p^*\geq p$ in $G$ and $i<\eta$,
    $$ j_{U^n}(p^*)\Vdash \la \check{\kappa}_0,\ldots, \check{\kappa}_{n-1} \ra \in j_{U^n}(\dot{X}_i). $$
    In particular, $X_i\in W$, as desired.

    Finally, we argue that $W$ is generated from $U^n$ in $V[G]$. Fix $X\in W$. Let $\dot{X}$ be a $\po$-name for $X$. Since $X\in W$, there exists $p\in G$ such that $j_{U^n}(p)\Vdash  \la \check{\kappa}_0,\ldots, \check{\kappa}_{n-1} \ra \in j_{U^n}(\dot{X})$. In particular,
    $$ A = \{ \vec{\nu} = \la \nu_0,\ldots, \nu_{n-1} \ra\in [\kappa]^n \colon p\Vdash \check{\vec{\nu}}\in \dot{X} \}\in U^n $$
    and $A\subseteq X$.
    \end{proof}

    The following Corollary extends the analysis done in the proof of Lemma \ref{Lem:LiftsOfPowersOfU}, by showing that the generic extension $V\subseteq V[G]$ preserves normal measures and their powers.
    
    \begin{corollary}\label{Cor:PreservingMeasureNormnalityAndPowers}
        Assume the hypotheses of Lemma \ref{Lem:LiftsOfPowersOfU}. Let $U\in V$ is a normal measure on $\kappa$, and let $U^*\in V[G]$ be the measure generated from $U$ in $V[G]$. Then $U^*$ is normal. Furthermore, for every $n<\omega$, $(U^*)^n$ is the measure generated in $V[G]$ from $U^n$. 
    \end{corollary}
    \begin{proof}
        We will continue to use the same notations from the proof of Lemma \ref{Lem:LiftsOfPowersOfU}. We first argue that the measure $U^*$ generated from $U$ in $V[G]$ is normal. Let $f\in V[G]$ be a regressive function on $\kappa$. Fix a $\po$-name $\dot{f}$ for it, and a condition $p\in G$ which forces that $\dot{f}$ is regressive. Let $H\subseteq j_{U}(\po)$ be the generic filter generated from $j_{U}[G]$. Since $j_U(p)\in H$, $H$ meets the set of conditions which decide $j_U(\dot{f})(\kappa)$. In particular, there exist an ordinal $\alpha$ and a condition of the form $j_U(q)$ for $q\in G$ such that
        $$ j_U(q)\Vdash j_U(\dot{f})(\check{\kappa}) = \check{\alpha}. $$
        Since $q\in G$, $q$ is compatible with $p$. Hence $j_U(q)$ and $j_U(p)$ are compatible, and the letter forces that $j_U(\dot{f})(\check{\kappa})<\check{\kappa}$. It this follows that $\alpha<\kappa$, and thus 
        $$ \{ \xi<\kappa \colon q\Vdash \dot{f}(\check{\xi}) = \check{\alpha} \}\in U. $$
        Since $q\in G$ and $U^*\supseteq U$, we deduce that, in $V[G]$,  
        $$ \{\xi<\kappa \colon f(\xi) = \alpha \}\in U^*, $$
        as desired. 
        
        Finally, fix $n<\omega$. Note that $U^n\subseteq (U^*)^n$ since $U\subseteq U^*$. It thus follows that $(U^*)^n$ is the measure generated from $U^n$ in $V[G]$. 
    \end{proof}

    \begin{remark}
        In the proofs of Lemma \ref{Lem:LiftsOfPowersOfU} and Corollary \ref{Cor:PreservingMeasureNormnalityAndPowers} we chose to avoid describing the ultrapower embedding  associated with $U^*$ in $V[G]$. To complete the picture, let us briefly describe it.
        
        As before, denote by $H\in V[G]$ the filter on $j_U(\po)$ generated from $j_U[G]$. Since $H$ is generic over $M_U$ and $j_U[G]\subseteq H$, we can apply Silver's lifting criterion to lift $j_U\colon V\to M_U\simeq \text{Ult}(V, U)$ to an elementary embedding  $j^*\colon V[G]\to M_U[H]$. Note that 
        $$U^*=\{ X\subseteq \kappa \colon \kappa\in j^*(X)  \},$$ 
        since both sides are ultrafilters in $V[G]$ that extend $U$, and $U$ already generates a measure in $V[G]$.\footnote{This, in particular, shows that $U^*$ is normal, and can replace the argument from the proof of Corollary \ref{Cor:PreservingMeasureNormnalityAndPowers}.} 

        Let us verify that $j^*=j^{V[G]}_{U^*}$, where $j^{V[G]}_{U^*}$ is the ultrapower embedding associated with $U^*$ over $V[G]$. Indeed, every element of $M_U[H]$ is of the form $j^*(f)(\kappa)$ for some $f\colon\kappa\to V[G]$ in $V[G]$.\footnote{This is a routine argument: Let $x=\tau_H\in M_U[H]$, where $\tau\in M_U$ is a $j_U(\po)$-name. Write $\tau=j_U(F)(\kappa)$ for some $F\colon\kappa\to V$. By \L o\'s's theorem, after modifying $F$ on a $U$-null set, we may assume that each $F(\xi)$ is a $\po$-name. Define $f\colon\kappa\to V[G]$ by
        $f(\xi)=F(\xi)_G$. Then
         $j^*(f)(\kappa)=\tau_H=x$.} It thus follows that the embedding $k\colon M^{V[G]}_{U^*}\simeq \text{Ult}(V[G],U^*)\to M_U[H]$, defined by $k([f]_{U^*})=j^*(f)(\kappa)$, is surjective, and hence an isomorphism.

        Finally, fix $n<\omega$, and let us address the measure in $V[G]$ generated from $U^n$. Denote it by $(U^n)^*$, and recall that we already established that $(U^n)^* = (U^*)^n$. In particular, we deduce that $j_{(U^n)^*}\colon V[G]\to \text{Ult}(V[G], (U^n)^*)$ coincides with $j_{(U^*)^n}$, which is the $n$-th iterated ultrapower of the normal measure $U^*$ over $V[G]$. 
    \end{remark}
    
    Let us finally proceed to the proof of the main result in this section.
    
    \begin{proof}[Proof of Theorem \ref{Thm:ForcingOverLU}]
        Assume $V = L[U]$, and $U$ is the unique normal measure in $V$. We argue that $V[G]$ is a `Kunen-like' model, in the sense that it has a unique measurable cardinal $\kappa$, a unique normal measure $U^*\in V[G]$ on $\kappa$, and every $\sigma$-complete ultrafilter is Rudin-Keisler equivalent to $(U^*)^n$ for some $n<\omega$. Such a model can be easily seen to be a model of {\rm UA}.

        We will make use of the following property of $L[U]$: For every $W\in V[G]$ which is a $\sigma$-complete ultrafilter, $W\cap L[U]\in L[U]$. This implies that no cardinal other than $\kappa$ can be measurable in $V[G]$.

        We already proved in Lemma \ref{Lem:LiftsOfPowersOfU} and Corollary \ref{Cor:PreservingMeasureNormnalityAndPowers} that $\kappa$ remains measurable in $V[G]$. Thus $\kappa$ is the unique measurable cardinal. Also, $U$ lifts to a normal measure $U^*$ in $V[G]$, and for every $n<\omega$, $U^n$ generates the measure $(U^*)^n$ in $V[G]$. It remains to show that, up to Rudin-Keisler equivalence, those are the only measures.

        Assume that $W\in V[G]$ is a $\sigma$-complete ultrafilter on an ordinal $\eta$. By the above property of $L[U]$, $W\cap L[U]\in L[U]$. In particular, for some $n<\omega$, $W\cap L[U] \equiv_{RK} U^n$. Let $f\colon \eta\to [\kappa]^n$ be an injection witnessing the Rudin-Keisler equivalence, in the sense that for every $X\subseteq [\kappa]^n$ in $L[U]$,
        $$X\in (U^*)^n \iff f^{-1}[X]\in W\cap L[U].$$ 
        It follows that $f$ witnesses in $V[G]$ that $W\equiv_{RK} (U^*)^n$.

    \end{proof}

\begin{corollary}\label{Cor:UA+Meas+VneqHODX}
    The {Ultrapower Axiom} is consistent with the existence of a measurable cardinal $\kappa$ such that, for every $X\in V_\kappa$, $V\neq {\rm{HOD}}_{X}$.
\end{corollary}

Prior to the proof, let us recall the standard notion of weakly homogeneous forcing and its application to ${\rm HOD}$.

\begin{definition}
A forcing notion $\po$ is \textit{weakly homogeneous} if for every $p,q\in\mathbb \po$, there exists an automorphism $\pi$ of $\mathbb \po$ such that $\pi(p)$ is compatible with $q$.    
\end{definition}

\begin{lemma}[Folklore]\label{Lem:HomogeneityAndHOD}
    Suppose that $\po\in V$ is a weakly homogeneous poset, $G\subseteq\po$ is generic over $V$, and $X\in V$. Then
    $$(\mathrm{HOD}_X)^{V[G]}\subseteq V.$$
\end{lemma}

\begin{proof}[Proof of Corollary \ref{Cor:UA+Meas+VneqHODX}]
    Assume $V = L[U]$. Fix a discrete unbounded subset $I \subseteq \kappa$ (for example, $I$ is the set of inaccessible cardinals which are not limits of inaccessible cardinals). Define a discrete product forcing $\po = \prod_{\alpha\in I} \qo_\alpha$ where, for every $\alpha\in I$, $\qo_\alpha = \text{Add}(\alpha,1)$. Let $G\subseteq \po$ be generic over $V$. By Theorem \ref{Thm:ForcingOverLU}, $\kappa$ remains measurable in $V[G]$, and $V[G]$ satisfies the Ultrapower Axiom. 
    Thus, it  remains to show that for every $X\in V[G]$ with $\text{rank}(X)<\kappa$, $V[G]\neq ({\rm{HOD}}_{X})^{V[G]}$. Fix such $X$. Since $\text{rank}(X)<\kappa$, there exists $\gamma\in I$ such that $X\in V[G_{\gamma}]$, where $G_\gamma = G\uhr \po_\gamma$ is the initial segment of the generic up to $\gamma$. We will show that $\po\setminus \gamma$ is a weakly homogeneous forcing. This suffices as it implies, by Lemma \ref{Lem:HomogeneityAndHOD}, that $(\rm{HOD}_X)^{V[G]}\subseteq V[G_\gamma]$. Since $I$ is unbounded in $\kappa$, the forcing $\po\setminus \gamma$ adds a new function $g\colon I\setminus \gamma\to V_\kappa$ which is generic over $V[G_{\gamma}]$ (more specifically, $g(\beta)$ for $\beta\in I\setminus \gamma$ is the Cohen generic added to $\qo_\beta$). In particular, $g
    \notin (\rm{HOD}_X)^{V[G]}$. Hence $V[G]\neq (\rm{HOD}_X)^{V[G]}$.

    To prove weak homogeneity, suppose that $p,q\in \prod_{\alpha\in I\setminus \gamma}\qo_\alpha$.  For every $\alpha\in I$, let $\Delta_{\alpha} = \{ \xi\in \dom(p(\alpha))\cap \dom(q(\alpha)) \colon p(\alpha)(\xi)\neq q(\alpha)(\xi)\}$. Define an automorphism $\sigma_\alpha$ of $\text{Add}(\alpha,1)$ by flipping the coordinates on $\Delta_{\alpha}$, namely, for every $s\in \text{Add}(\alpha,1)$,
    $$ \sigma_\alpha(s)(\xi) =  
    \begin{cases}
        1-s(\xi), & \text{if } \xi\in\Delta_\alpha\cap \dom(s),\\
        s(\xi), & \text{if } \text{otherwise.}
        \end{cases}$$
    Then $\sigma\colon \po\setminus \gamma \to \po\setminus \gamma$, defined by setting, for every $r\in \po\setminus \gamma$,
    $$ \sigma(r)(\alpha) = \sigma_\alpha(r(\alpha)) $$
    is an automorphism such that $\sigma(p), q$ are compatible. \end{proof}

\section{A strong cardinal}\label{Sec:Strong}

Throughout this section, we work over the minimal canonical inner model $L[\vec{E}]$ with a strong cardinal $\kappa$. We first recall the definition of a coherent nonoverlapping sequence of extenders and some basic facts about canonical inner models constructed from such sequences. Our presentation follows Mitchell's treatment in \cite{MitchellInnerModel} and is also influenced by the presentation in \cite{aptercummings2023normalmeasuresonlargecards}. Readers familiar with this material may skip ahead to Definition \ref{Def:SuitablySpaced}.

\begin{definition}(Mitchell)
A \textit{{coherent nonoverlapping sequence of extenders}} is a function
$\vec E$ with $\dom(\vec E)\subseteq {\rm On}\times{\rm On}$
satisfying the following conditions.

\begin{enumerate}
    \item There is a function
    $o^{\vec E}\colon {\rm On} \to ({\rm On}\cup \{ \infty \})$ such that
    $$    \dom(\vec E)  =   \{(\alpha,\beta)\colon \beta<o^{\vec E}(\alpha)\}. $$

    \item If $(\alpha,\beta)\in\dom(\vec E)$, then
    $\vec E(\alpha,\beta)$ is a total $(\alpha,\alpha+1+\beta)$-extender. Thus
    $$\text{crit}(E(\alpha,\beta))=\alpha$$
    and 
    $$\text{lh}(E(\alpha,\beta))\leq\alpha+1+\beta.$$
    \item If $(\alpha,\beta)\in \dom(\vec E)$ and
    $E=\vec E(\alpha,\beta)$, then
    $$
    j_E(o^{\vec E})(\alpha)=\beta
    \qquad\text{and}\qquad
    j_E(o^{\vec E})\restriction\alpha
    =
    o^{\vec E}\restriction\alpha,
    $$
    and whenever
    $(\alpha',\beta')<_{\text{lex}}(\alpha,\beta)$,\footnote{$<_{\text{lex}}$ denotes the lexicographic ordering on ${\rm On}\times{\rm On}$, so that
    $(\alpha_0,\beta_0)<_{\text{lex}}(\alpha_1,\beta_1)$ if and only if $\alpha_0<\alpha_1$ or $\bigl(\alpha_0=\alpha_1\text{ and }\beta_0<\beta_1\bigr)$.}
    $$\vec E(\alpha',\beta')  =  j_E(\vec E)(\alpha',\beta').$$

    \item If $\alpha_0<\alpha_1$ and
    $o^{\vec E}(\alpha_1)>0$, then
    $$    o^{\vec E}(\alpha_0)<\alpha_1.$$

\end{enumerate}

Conditions (1)--(3) express the coherence of the extender sequence,
while condition (4) is the nonoverlapping condition.
\end{definition}

Since $L[\vec{E}]$ is the minimal inner model having a strong cardinal $\kappa$, $o^{\vec E}(\kappa)=\infty$ and $o^{\vec E}(\alpha)<\kappa$ for every
$\alpha<\kappa$. Moreover, $\kappa$ is the largest ordinal $\alpha$ such that $o^{\vec E}(\alpha)>0$.

Throughout the section, we will assume that $E(\alpha,\beta)\in L[\vec{E}]$ for every $(\alpha,\beta)\in \dom(\vec{E})$.\footnote{This can be assumed by identifying each extender $E(\alpha,\beta) = \la E(\alpha,\beta)(a) \colon a\in [\alpha+1+\beta]^{<\omega} \ra$ with its trace to $L[\vec{E}]$,  $E'(\alpha,\beta) = \la E(\alpha,\beta)(a)\cap L[\vec{E}] \colon a\in [\alpha+1+\beta]^{<\omega} \ra$.}

\begin{definition}\label{Def:NormalIterates}
An embedding $j\colon L[\vec E]\to M$ is a \textit{normal linear
iteration by extenders on the $L[\vec E]$-sequence} if there are an
ordinal $\theta$, models
$$\langle M_\xi\colon \xi\leq\theta\rangle,$$
elementary embeddings
$$\langle j_{\xi,\eta}\colon \xi\leq\eta\leq\theta\rangle,$$
and extenders
$$\langle F_\xi\colon \xi<\theta\rangle$$
satisfying the following conditions.

\begin{enumerate}
    \item $M_0=L[\vec E]$.

    \item For every $\xi<\theta$, $F_\xi$ belongs to the extender
    sequence of $M_\xi$, and
    \[
    M_{\xi+1}=\Ult(M_\xi,F_\xi),
    \]
    where
    \[
    j_{\xi,\xi+1}\colon M_\xi\to M_{\xi+1}
    \]
    is the corresponding ultrapower embedding.

    \item The embeddings form a commuting system; that is,
    $j_{\xi,\xi}=\text{id}_{M_\xi}$ and
    \[
    j_{\xi,\zeta}
    =
    j_{\eta,\zeta}\circ j_{\xi,\eta}
    \]
    whenever $\xi\leq\eta\leq\zeta\leq\theta$.

    \item If $\lambda\leq\theta$ is a limit ordinal, then $M_\lambda$,
    together with the maps $j_{\xi,\lambda}$ for $\xi<\lambda$, is the
    direct limit of
    \[
    \langle M_\xi,j_{\xi,\eta}\colon \xi\leq\eta<\lambda\rangle.
    \]

    \item The critical points of the extenders used in the iteration are
    strictly increasing; that is,
    \[
    \text{crit}(F_\xi)<\text{crit}(F_\eta)
    \]
    whenever $\xi<\eta<\theta$.

    \item $M=M_\theta$ and $j=j_{0,\theta}$.
\end{enumerate}
\end{definition}

Notice that although $F_0$ belongs to the original extender sequence $\vec E$, at later stages the extender $F_\xi$ belongs to the extender sequence of the corresponding iterate $M_\xi$. Thus, when we say that the iteration uses extenders on the $L[\vec E]$-sequence, we include the images of extenders from $\vec E$ along the iteration.

\begin{remark}
In Definition \ref{Def:NormalIterates}, normality only requires that the
critical points of the extenders used in the iteration are strictly
increasing. Since the extender sequence of each iterate $M_\xi$ is
nonoverlapping, this implies the stronger property that
$$\text{lh}(F_\xi)<\text{crit}(F_\eta)$$
whenever $\xi<\eta<\theta$. Indeed, if
$F_\xi=\vec E^{M_\xi}(\alpha_\xi,\beta_\xi)$, then coherence gives
$$o^{\vec E^{M_{\xi+1}}}(\alpha_\xi)=\beta_\xi,$$
while non-overlap and
$\alpha_\xi<\text{crit}(F_{\xi+1})$ imply
$$\beta_\xi<\text{crit}(F_{\xi+1}).$$
Since $\text{crit}(F_{\xi+1})$ is a cardinal in $M_{\xi+1}$, it follows that
$$\alpha_\xi+1+\beta_\xi<\text{crit}(F_{\xi+1}).$$
Thus
$$\text{lh}(F_\xi)<\text{crit}(F_{\xi+1}),$$
and hence
$$\text{lh}(F_\xi)<\text{crit}(F_\eta)$$
for every $\xi<\eta<\theta$.

In particular, every iteration map following the ultrapower by $F_\xi$ fixes the generators of $F_\xi$.
\end{remark}

\begin{definition}[Mitchell]\label{Def:Supports}
    Assume that $j\colon L[\vec{E}]\to M$ is a normal linear iteration by extenders on the $L[\vec{E}]$-sequence and $x\in M$. A \textit{support} for $x$ is a system consisting of $n<\omega$, a function $h\colon ([\kappa]^{<\omega})^n\to L[\vec E]$ in $L[\vec E]$, and sequences 
    $$\vec{f} = \la f_0,\ldots,f_{n-1} \ra$$
    of functions in $L[\vec E]$,
    $$\vec{N} = \la N_0,\ldots,N_n \ra$$
    of iterates of $L[\vec{E}]$,
    $$\vec{i} = \la i_{\ell,m}\colon \ell\leq m\leq n \ra$$
    of elementary embeddings,
    $$\vec{F} = \la F_0,\ldots,F_{n-1} \ra$$
    of extenders, and
    $$\vec{a} = \la \vec{a}_0,\ldots,\vec{a}_{n-1} \ra$$
    of tuples, such that:
    \begin{enumerate}
        \item $N_0 = L[\vec{E}]$.
        
        \item For every $\ell<n$, 
        $$i_{0,\ell}(f_{\ell})(\vec{a}_0,\ldots,\vec{a}_{\ell-1})$$
        is a pair $\la \alpha_{\ell},\beta_{\ell}\ra$ such that
        $$F_{\ell} = \vec{E}^{N_{\ell}}(\alpha_{\ell},\beta_{\ell}).$$
        
        \item $\vec{a}_\ell$ is a finite tuple of generators of $F_\ell$. In particular,
        $$\vec{a}_{\ell}\in [(\alpha_{\ell}+1+\beta_{\ell})\setminus\alpha_{\ell}]^{<\omega}.$$
        
        \item For every $\ell<n$, $N_{\ell+1} = \Ult(N_\ell,F_\ell)$, and
        $$i_{\ell,\ell+1}\colon N_\ell\to N_{\ell+1}$$
        is the corresponding extender ultrapower embedding.
        
        \item The embeddings 
        $$\la i_{\ell,m}\colon \ell\leq m\leq n\ra$$
        form a commuting system.
        
        \item For every $\ell<m<n$,
        $$\alpha_{\ell}+1+\beta_{\ell}<\text{crit}(F_m).$$
        
        \item There exists a normal linear iteration $k\colon N_n\to M$ by extenders on the $\vec{E}^{N_n}$-sequence,  such that
        $$j=k\circ i_{0,n}$$
        and
        $$x=k\left(i_{0,n}(h)(\vec{a}_0,\ldots,\vec{a}_{n-1})\right).$$
    \end{enumerate}
\end{definition}

\begin{lemma}[The Support Lemma]\label{Lem:SupportLemma}
    Assume that $\vec{E}$ is a coherent nonoverlapping sequence of extenders. Let $j\colon L[\vec{E}]\to M$ be a normal linear iteration by extenders on the $L[\vec{E}]$-sequence. Then every $x\in M$ has a support.
\end{lemma}

\begin{proof}    
    The proof is by induction on the length $\zeta$ of the iteration $j\colon L[\vec E]\to M$. 

    Assume first that $\zeta$ is a limit ordinal. Fix $x\in M = M_\zeta$. Since $\zeta$ is limit, there exists $\zeta'<\zeta$ and $x'\in M_{\zeta'}$ such that $x= j_{\zeta', \zeta}(x')$. By induction, $x'$ has a support system $\la n, h, \vec{f} ,\vec{N}, \vec{i}, \vec{F},  \vec{a}\ra$. Let $k'\colon N_n\to M_{\zeta'}$ be such that $j_{\zeta'} = k'\circ i_{0,n}$ and $x' = k'\left( i_{0,n}(h)(\vec{a}_0, \ldots, \vec{a}_{n-1}) \right)$. Define $k\colon N_n\to M_\zeta$ by $k = j_{\zeta', \zeta} \circ k'$. Then $k$ is a normal linear iteration, $j_{\zeta} = k\circ i_{0,n}$. Furthermore,  $x = k\left( i_{0,n}(h)(\vec{a}_0, \ldots, \vec{a}_{n-1}) \right)$, and thus $\la n,h,\vec{f} ,\vec{N}, \vec{i}, \vec{F},  \vec{a}\ra $ remains a support system for $x\in M_\zeta$. 

   We thus proceed to the case where $\zeta=\zeta'+1$ is a successor ordinal. Write $x=j_F(g)(a)$, where $F$ is an extender in $\vec E^{M_{\zeta'}}$, $g\in M_{\zeta'}$ is a function, and $a\in[\text{lh}(F)]^{<\omega}$. Write $F=\vec E^{M_{\zeta'}}(\alpha,\beta)$. Since the tuple $\la\alpha,\beta,a, g\ra$ belongs to $M_{\zeta'}$, by the induction hypothesis it admits a support system $\la n,h,\vec f,\vec N,\vec i,\vec F,\vec a\ra$. Let $k'\colon N_n\to M_{\zeta'}$ be its final factor map. 
   Thus $k'$ is a normal linear iteration,
   $$j_{\zeta'}=k'\circ i_{0,n},$$
    and, writing $$\la\alpha',\beta',a', g'\ra=
    i_{0,n}(h)(\vec a_0,\ldots,\vec a_{n-1}),$$
    we have $k'( \alpha' ,\beta', a',g' )= (\alpha,\beta, a, g)$.
    Define a function $f_n$ in $V$ such that $f_n(\vec\nu_0,\ldots,\vec\nu_{n-1})$ equals the first two
    coordinates of $h(\vec\nu_0,\ldots,\vec\nu_{n-1})$. Let $F_n:=\vec E^{N_n}(\alpha',\beta')$, $a_n= a'$, $N_{n+1}:=\Ult(N_n,F_n)$, and let $i_{n,n+1}\colon N_n\to N_{n+1}$ be the corresponding ultrapower embedding. 
    
    Since $k'(F_n)=F$, define
    $$ k\colon \Ult(N_n,F_n)\to \Ult(M_{\zeta'},F)$$
    by
    $$ k\left(j_{F_n}(u)(b)\right)= j_F(k'(u))(k'(b)).$$
    It's not hard to see that $k$ is an elementary embedding and
    $$k\circ j_{F_n}=j_F\circ k'. $$
    Moreover, by copying the normal linear iteration $k'$ along $j_{F_n}$,  the map $k$ is itself a normal linear iteration by extenders on the  $\vec E^{N_{n+1}}$-sequence
    (see \cite[Lemma 4.3.1]{ZemanBook}).\footnote{The fact that the copied map coincides with the definition of the embedding $k$ as given above follows by checking that the copied map maps each tuple $b\in [\text{lh}(F_n)]^{<\omega}$ to $k'(b)$. This is done simultaneously with the inductive copying process. We refer the interested reader to \cite[Lemma 4.3.1]{ZemanBook}.}
    
    \[
    \begin{tikzcd}[column sep=huge, row sep=large]
    N_n
        \arrow[r, "k'"]
        \arrow[d, "j_{F_n}"']
    &
    M_{\zeta'}
        \arrow[d, "j_F = j_{k'(F_n)}"]
    \\
    N_{n+1}=\Ult(N_n,F_n)
        \arrow[r, "k"'{name=K}]
    &
    M_{\zeta}=\Ult(M_{\zeta'},F)
    \arrow[from=K, phantom]
    \end{tikzcd}
    \]
At this point, the finite iteration
$$N_0\longrightarrow\cdots\longrightarrow N_n
\overset{j_{F_n}}{\longrightarrow}N_{n+1}$$
need not be normal, since we don't know that
$$\alpha_{n-1}+1+\beta_{n-1}<\alpha'=\text{crit}(F_n).$$

Let us first address the favorable case in which it is normal. In this case, define $h'$ by letting $h'(\vec\nu_0,\ldots,\vec\nu_n)$ be the value at $\vec\nu_n$ of the function given by the fourth coordinate of $h(\vec\nu_0,\ldots,\vec\nu_{n-1})$, whenever this is defined, and define it arbitrarily otherwise. Then, since
$j_{F_n}$ fixes $\vec a_0,\ldots,\vec a_{n-1}$,
$$i_{0,n+1}(h')(\vec a_0,\ldots,\vec a_n)=j_{F_n}(g')(a').$$
Consequently,
\begin{align*}
k\left(
i_{0,n+1}(h')(\vec a_0,\ldots,\vec a_n)
\right)
&=
k\left(j_{F_n}(g')(a')\right)\\
&=
j_F(k'(g'))(k'(a'))\\
&=
j_F(g)(a)\\
&=x.    
\end{align*}
Thus, let us assume that the iteration 
$$N_0\longrightarrow\cdots\longrightarrow N_n
\overset{j_{F_n}}{\longrightarrow}N_{n+1}$$
is not normal. We therefore need to rearrange this finite iteration, by commuting $F_n$ with extenders $F_\ell$ whose indices are above $\alpha'$. 

Let us demonstrate this change in the case where only the last extender $F_{n-1}$ has to be switched with $F_n$ in order to achieve normality (the general case in which more extenders have to be switched is taken care of similarly). Thus, assume that $\alpha_{n-2}<\alpha'<\alpha_{n-1}$. By coherence, $F_n$ already lies on the extender sequence of $N_{n-1}$, and $j_{F_{n-1}}$ fixes $F_n$. 

We define a new support system
$$\la n+1,h^*,\vec f^*,\vec N^*,\vec i^*,\vec F^*,\vec a^*\ra.$$
Below stage $n-1$, retain all the previous support data; namely, let
$$
\vec f^*\restriction(n-1)=\vec f\restriction(n-1),\qquad
\vec F^*\restriction(n-1)=\vec F\restriction(n-1),
$$
$$
\vec a^*\restriction(n-1)=\vec a\restriction(n-1),
\qquad\text{and}\qquad
\vec N^*\restriction n=\vec N\restriction n,
$$
and let $i_{\ell,r}^*=i_{\ell,r}$ for all $\ell\leq r\leq n-1$. 

If necessary, enlarge the tuples $\vec{a}_0,\ldots, \vec{a}_{n-2}$ by adding to each tuple finitely many more ordinals, such that $\la \alpha', \beta' \ra$ are already supported; namely, there exists a function $f_{n-1}^*\in L[\vec E]$ such that
$$i_{0,n-1}(f_{n-1}^*)(\vec a_0,\ldots,\vec a_{n-2})=\la\alpha',\beta'\ra.$$
Also define
$$f^*_n(\vec\nu_0,\ldots,\vec\nu_{n-1})=f_{n-1}(\vec\nu_0,\ldots,\vec\nu_{n-2}).$$

Let $F^*_{n-1}:=F_n$, $\vec a_{n-1}^*:=a'$  , $N^*_{n} = \Ult(N^*_{n-1}, F^*_{n-1})$, and $i_{n-1,n}^*\colon N^*_{n-1}\to N^*_{n}$ be the corresponding ultrapower embedding.

Finally, let $F^*_n:=j_{F_n}(F_{n-1})$,\footnote{Note that $i_{0,n}^*(f_n^*)(\vec a_0^*,\ldots,\vec a_{n-1}^*)=
j_{F_n}(\la\alpha_{n-1},\beta_{n-1}\ra)$, so $f^*_n$ indeed represents the extender $F^*_n=j_{F_n}(F_{n-1})$ at the new stage $n$.} $\vec a_n^*:=j_{F_n}(\vec a_{n-1})$, $N^*_{n+1} = \Ult( N^*_{n}, F^*_{n} )$, and $i_{n,n+1}^*\colon N^*_{n}\to N^*_{n+1}$ the corresponding ultrapower embeddings. 

Define $h^*\in N$ by letting $h^*(\vec\nu_0,\ldots, \vec{\nu}_{n-1}, \vec\nu_n)
$ be the value at $\vec\nu_{n-1}$ of the function given by the fourth coordinate of $h(\vec\nu_0,\ldots,\vec\nu_{n-2},\vec\nu_n),$
whenever this is defined, and define it arbitrarily otherwise. 

Note that $N_{n+1}^* = N_{n+1}$, and
$$i_{n,n+1}^*\circ i_{n-1,n}^*=j_{F_n}\circ j_{F_{n-1}}.$$
It thus follows that  
$$j_\zeta=k\circ i_{0,n+1}=k\circ i_{0,n+1}^*.$$

Finally, by the definitions above,
$$i_{0,n+1}^*(h^*)(\vec a_0^*,\ldots,\vec a_n^*) = j_{F_n}(g')(a').$$
Consequently,
\begin{align*}
k\left(
i_{0,n+1}^*(h^*)
(\vec a_0^*,\ldots,\vec a_n^*)
\right)
&=
k\left(j_{F_n}(g')(a')\right)\\
&=
j_F(k'(g'))(k'(a'))\\
&=
j_F(g)(a)\\
&=
x.
\end{align*}
\end{proof}

\begin{definition}\label{Def:SuitablySpaced}
    Assume that $V = L[\vec{E}]$ where $\vec{E}$ is a coherent sequence of nonoverlapping extenders. 
    \begin{enumerate}
        \item Suppose that $I\subseteq \kappa$ is an unbounded discrete set of inaccessible cardinals. We say that $I$ is \textit{suitably spaced} if 
        $$I\subseteq \{ \alpha<\kappa \colon \forall \alpha'<\alpha, o^{\vec E}(\alpha')<\alpha \}.\footnote{In particular, $\min(I\setminus (\alpha+1))> o^{\vec E}(\alpha)$ for every $\alpha<\kappa$,}$$
        and for every $\alpha\leq \kappa$ with $o^{\vec E}(\alpha)> 0$, $I$ is unbounded in $\alpha$.\footnote{In particular, by discreteness, $I\subseteq \{ \alpha<\kappa \colon o^{\vec{E}}(\alpha) =0 \}$.}
        \item We say that a discrete product forcing $\po = \prod_{\alpha\in I}\qo_\alpha$ is \textit{suitably spaced} if $I$ is suitably spaced.
    \end{enumerate}
\end{definition}

\begin{remark}\label{Remark:ExistenceOfSuitablySpacedSets}
To see that Definition \ref{Def:SuitablySpaced} is not vacuous, let
$$S=\{\alpha<\kappa \colon \forall \alpha'<\alpha,\ o^{\vec E}(\alpha')<\alpha\}.$$
The fact that $\vec{E}$ is nonoverlapping implies that $S\cap\lambda$ is club in $\lambda$ for every $\lambda\leq\kappa$ with $o^{\vec E}(\lambda)>0$. Therefore, the set $I$ of inaccessible cardinals in $S$ which are not limits of inaccessible cardinals in $S$ is suitably spaced.    
\end{remark}

Let us now argue that extenders from the $L[\vec{E}]$-sequence admit canonical lifts after a suitably spaced discrete product forcing has been done.

\begin{lemma}[Canonical lifts of extenders on the {$L[\vec E]$}-sequence]\label{Lem:CanonicalLiftOfAnExtender}
    Assume that $V = L[\vec{E}]$ where $\vec{E}$ is a coherent sequence of nonoverlapping extenders. Let $\po = \prod_{\alpha\in I} \qo_\alpha$ be a suitably spaced discrete product forcing that satisfies the hypotheses of Definition \ref{Def:discrete-product-forcings}.  Fix $\alpha\leq \kappa$ with $o^{\vec E}(\alpha)>0$ and $\beta<o^{\vec E}(\alpha)$. Let $F = E(\alpha,\beta)$, and denote by $j_F\colon V\to M$ the corresponding ultrapower embedding. 
    
    Let $G\subseteq \po$ be generic over $L[\vec E]$. Then $j_F[G]$ generates a $j_F(\po)$-generic set over $M$, and $j_F$ canonically lifts to an embedding $j^*_F\colon V[G]\to M[j_F[G]]$.
    \end{lemma}

    \begin{proof}
        Let us argue that $j_F[G]$ generates a $j_F(\po)$-generic over $M$. Note that since $I$ is suitably spaced and $o^{\vec E}(\alpha)>0$, $\alpha = \sup(I\cap \alpha)$.
        
        Fix $D\in M$ dense open. Let $a\in [(\alpha+1+\beta)\setminus \alpha]^{<\omega}$ be a finite set of generators for $F$, and $d\colon [\alpha]^{<\omega}\to V$ be a function in $V$, such that $D = j_F(d)(a)$. We can assume that for every $\vec{\nu}\in [\alpha]^{<\omega}$, $d(\vec{\nu})\subseteq \po$ is dense open. Let
    \begin{align*}
        E = \{ p\in \po \colon &\text{for every }\xi<\alpha \text{ with }\xi= \sup(\xi\cap I) \text{ and for every }\\ &\vec{\nu}\in [ \xi+1+o^{\vec{E}}(\xi) ]^{<\omega}, \text{ the set}\\
        &\{ r\in \po_\xi \colon r{}^\frown p\setminus  \xi \in d(\vec{\nu})\}\subseteq \po_\xi\\
        &\text{is a dense open subset of }\po_{\xi} \text{ above } p\uhr \xi \}.
    \end{align*}
    Note that for every $\xi<\alpha$,  $\xi+1+o^{\vec{E}}(\xi)<\alpha$ by nonoverlapping. In particular, $d(\vec{\nu})$ is defined for each $\vec{\nu}\in [\xi+1+o^{\vec E}(\xi)]^{<\omega}$. To simplify the notation, denote for every $\xi<\alpha$, $\xi':=\xi+1+o^{\vec E}(\xi)$.
    
    \begin{claim}\label{Claim:TheSetE}
        $E\subseteq \po$ is a dense open subset.
    \end{claim}

    \begin{proof}
        Fix $p\in \po$. Since $o^{\vec{E}}(\alpha)>0$, the forcing $\po_\alpha = \prod_{\alpha'\in I\cap \alpha}\qo_{\alpha'}$ is by itself a suitably spaced discrete product forcing (with respect to the coherent sequence $\vec{E}\uhr \alpha$). 
        
     We first argue that there exists $t\in \po\setminus\alpha$, $t\geq p\setminus\alpha$, such that for every $\vec{\nu}\in[\sup(I\cap \alpha)]^{<\omega}$, the set
$$d_\alpha(\vec{\nu}):=\{q\in\po_\alpha\colon q{}^\frown t\in d(\vec{\nu})\}$$
is dense open in $\po_\alpha$ above $p\uhr\alpha$.\footnote{Note that in the case where $\alpha = \kappa$ this is trivial.} To see this, enumerate $\langle\langle q_i,\vec{\nu}_i\rangle\colon i<\zeta\rangle$ the set $\po_\alpha\times[\sup(I\cap \alpha)]^{<\omega}$, where $\zeta  = \left|\po_\alpha\times[\sup(I\cap \alpha)]^{<\omega} \right|$. We argue that the forcing $\po\setminus\alpha$ is  $\zeta^+$-closed. Indeed, if $\alpha\notin I$ the forcing $\po\setminus \alpha$ is more than $\alpha^+$-closed, and $\zeta\leq \alpha^+$. If $\alpha\in I$, by discreteness, the set $I\cap \alpha$ is bounded in $\alpha$ and thus $\zeta<\alpha$. It follows that $\po\setminus \alpha$ is more than $\zeta$-closed. Overall, we may therefore construct an increasing sequence of conditions $\langle t_i\colon i<\zeta\rangle$ with an upper bound $t$, such that for every $i<\zeta$, there is $q_i'\geq q_i$ in $\po_\alpha$ satisfying
$$q_i'{}^\frown t_i\in d(\vec{\nu}_i).$$
Since each $d(\vec{\nu})$ is open and $t\geq t_i$, it follows that
$q_i'{}^\frown t\in d(\vec{\nu}_i)$. Hence each $d_\alpha(\vec{\nu})$ is dense above $p\uhr\alpha$. Its openness follows immediately from the openness of $d(\vec{\nu})$.
        
        Next, for every $\xi<\alpha$ with  $\xi\notin I$, let
        \begin{align*}
             \mathfrak{d}(\xi) = \{ q\in \po_\alpha \colon &\text{for every } \vec{\nu}\in [\xi']^{<\omega},\text{ the set }  \{ r\in \po_\xi \colon r{}^\frown q\setminus \xi{} \in d_\alpha(\vec{\nu}) \}\subseteq \po_{\xi} \\
             &\text{is a dense open subset of }\po_\xi \text{ above }q\uhr \xi\}. 
        \end{align*}
        For every other value of $\xi<\alpha$, let $\mathfrak{d}(\xi)\subseteq \po_\alpha$ be an arbitrary dense open subset.
        
        We argue that $\la \mathfrak{d}(\xi)\colon \xi<\alpha\ra$ is a sequence of dense open subsets of $\po_\alpha$ above $p\uhr\alpha$. Fix $\xi<\alpha$ with $\xi \notin I$, and let $\langle\la r_i,\vec{\nu}_i\rangle\colon i<\zeta\ra$ enumerate $\po_\xi\times[\xi']^{<\omega}$, where $\zeta=|\po_\xi\times [\xi']^{<\omega}|\leq {\xi'}^+$. Fix $q\in\po_\alpha$ extending $p\uhr\alpha$. Construct an increasing sequence $\langle s_i\colon i<\zeta\rangle\subseteq\po_\alpha\setminus\xi$ of extensions of $q\setminus\xi$, together with an upper bound $s$, such that for every $i<\zeta$ there is some $r\geq r_i$ in $\po_\xi$ satisfying $r{}^\frown s_i\in d_\alpha(\vec{\nu}_i)$. Note that upper bounds can be taken at limit stages because $\po_\alpha\setminus\xi$ is $\min(I\setminus (\xi+1))$-closed. In particular, by nonoverlapping, $\po_\alpha\setminus \xi$ is ${\xi'}^+$-closed. It follows, using the openness of each $d_\alpha(\vec{\nu})$, that
        $(q\uhr\xi){}^\frown s\in\mathfrak d(\xi)$.
        
        Apply the Discrete Fusion Lemma (Lemma \ref{Lem:DiscreteFusion}) on the sequence $\la \mathfrak{d}(\xi) \colon \xi<\alpha \ra$ to find a condition $q\geq p\uhr \alpha$ such that for every $\xi<\alpha$ with $\xi = \sup(I\cap \xi)$,
        $$ \{ r\in \po_\xi \colon r{}^\frown q\setminus \xi \in \mathfrak{d}(\xi) \}\subseteq \po_\xi $$
        is a dense open subset of $\po_\xi$ above $q\uhr \xi$. 
        
        Let $p^* = q{}^\frown t$. We argue that $p\leq p^*\in E$. Indeed, fix $\xi<\alpha$ with $\xi =\sup(I\cap \xi)$ and $\vec{\nu}\in [\xi']^{<\omega}$. It remains to show that the set $\{r\in \po_\xi \colon r{}^\frown p^*\setminus \xi \in d(\vec{\nu}) \}$ is a dense open subset of $\po_\xi$. 
        
        Fix $r\in \po_\xi$. By the choice of $q$, there exists $r'\geq r$ in $\po_\xi$ such that $r'{}^\frown q\setminus \xi\in \mathfrak{d}(\xi)$. By the definition of $\mathfrak{d}(\xi)$, there exists $r''\geq r$ in $\po_\xi$ such that ${r''}{}^\frown q\setminus \xi\in d_\alpha(\vec{\nu})$. Thus, by the definition of the set $d_\alpha(\vec{\nu})$, ${r''}{}^\frown q\setminus \xi{}^\frown t\in d(\vec{\nu})$, namely ${r''}{}^\frown p^*\setminus \xi \in d(\vec{\nu})$, as desired.
        \end{proof}
        Since $E\subseteq \po$ is dense open, there exists some $p\in G\cap E$. Thus, in $V$, for every $\xi<\alpha$ with $\xi = \sup(I\cap \xi)$ and for every $\vec{\nu}\in [\xi' ]^{<\omega}$, the set
        $$ \{  r\in \po_\xi \colon r{}^\frown p\setminus \xi\in d(\vec{\nu}) \} $$
        is a dense open subset of $\po_\xi$ above $p\uhr \xi$. 

        Moving to $M$, recall that the set  $j_F(I)\cap\alpha=I\cap\alpha$ is unbounded in $\alpha$.   Moreover, $a\subseteq \alpha+1+\beta$, namely 
        $$a\in\left[ \alpha+1+o^{{\vec{E}}^M}(\alpha)\right]^{<\omega}.$$
        Thus, by elementarity, in $M$,
        $$D^* := \{ r\in j_F(\po)_\alpha \colon r{}^\frown j(p)\setminus \alpha \in j_F(\vec{d})(a) \}\subseteq j_F(\po)_\alpha $$
        is a dense open subset of $j_F(\po)_\alpha$ above $j_F(p)\uhr \alpha$. Note that $D = j_F(d)(a)$. Also, $j_F(\po)_\alpha = \po_\alpha$ and $j_F(p)\uhr \alpha = p\uhr \alpha$. Since $G\uhr \alpha$ is already $\po_\alpha$-generic over $V$, $D^*\in M\subseteq V$ and $D^*\subseteq \po_\alpha$ is dense open, there exists a condition $p'\in G$ such that $p'\uhr \alpha\in D^*$. In particular, 
        $$ (j_F(p')\uhr \alpha){}^\frown (j_F(p)\setminus \alpha) = (p'\uhr \alpha){}^\frown (j_F(p)\setminus \alpha)  \in D. $$
        Finally, pick $p^*\in G$ that extends both $p, p'$. Since $D$ is open and $j_F(p^*)\geq (j_F(p')\uhr \alpha){}^\frown (j_F(p)\setminus \alpha)$, we deduce that $j_F(p^*)\in D$. 
    \end{proof}

\begin{corollary}\label{Cor:PreserviungStrongness}
    Assume that $V = L[\vec{E}]$ where $\vec{E}$ is a coherent sequence of nonoverlapping extenders. Let $\po = \prod_{\alpha\in I} \qo_\alpha$ be a suitably spaced discrete product forcing that satisfies the hypotheses of Definition \ref{Def:discrete-product-forcings}. Let $G\subseteq \po$ be generic over $L[\vec{E}]$. Then $\kappa$ remains a strong cardinal in $V[G]$. 
\end{corollary}

\begin{proof}
    We argue that the strongness of $\kappa$ in $V[G]$ is witnessed by the canonical lifts of extenders on $\kappa$ from the $L[\vec{E}]$-sequence. Thus, fix $\lambda>2^{\kappa^+}$ and let $\beta\in \text{Ord}$ be such that $F:=E(\kappa, \beta)$ witnesses that $\kappa$ is $\lambda$-strong in $L[\vec{E}]$.  Let $j^*_F\colon L[\vec E][G]\to M_F[j_F[G]]$ be the canonical lift of $j_F$ as in the proof of Lemma \ref{Lem:CanonicalLiftOfAnExtender}. We argue that $j^*_F$ witnesses the $\lambda$-strongness of $\kappa$ in $V[G]$. In other words, $(V_\lambda)^{V[G]}\subseteq M_F[j_F[G]]$. To see this, fix $x\in V[G]$ of rank less than $\lambda$. Since $\lambda$ was chosen sufficiently large, we can fix a $\po$-name $\tau$ for $x$ such that $\tau\in V_\lambda$. In particular, $\tau\in M_F$, and thus $x = (\tau)_G\in M_F[j_F[G]]$.
\end{proof}

\begin{theorem}[Canonical lift of a normal iteration]\label{Thm:LiftingNormalIteratesAfterDiscreteProcuctForcing}
    Assume that $V = L[\vec{E}]$ where $\vec{E}$ is a coherent sequence of nonoverlapping extenders. Let $\po = \prod_{\alpha\in I} \qo_\alpha$ be a suitably spaced discrete product forcing that satisfies the hypotheses of Definition \ref{Def:discrete-product-forcings}. Let $G\subseteq \po$ be generic over $L[\vec{E}]$.
    
    Let $j\colon L[\vec{E}]\to M$ be a normal linear iteration by extenders on the $L[\vec{E}]$-sequence. Let $H\subseteq j(\po)$ be the filter generated from $j[G]$. Then $H$ is generic over $M$, and $j$ lifts to $j^*\colon L[\vec{E}][G]\to M[H]$. Moreover, $j^*$ is by itself a normal iteration of $L[E][G]$ by extenders.
\end{theorem}

\begin{proof}

    Let us first argue that $j[G]$ generates a $j(\po)$-generic set over $M$. We prove this by first reducing the proof of the case where $j$ is an infinite iteration to the case where $j$ is finite.
    
    Suppose that $D\in M$ is a dense open subset of $j(\po)$. We would like to show that for some $p\in G$, $j(p)\in D$. Fix a support system $n, h, \vec{f}, \vec{N}, \vec{i}, \vec{F}, \vec{a}$ for $D$ as in Definition \ref{Def:Supports}. Denote $D' = i_{0,n}(h)( \vec{a}_0,\ldots, \vec{a}_{n-1} )$ so that the connecting embedding $k\colon N_n\to M$ satisfies $k(D') = D$. In particular, $D'\subseteq i_{0,n}(\po)$ is dense and open. Since $i_{0,n}$ is a finite normal iteration by extenders on the $L[\vec E]$-sequence, for which we assume the theorem applies,  there exists $p\in G$ such that $i_{0,n}(p)\in D'$. In particular, $j(p) = k(i_{0,n}(p)) \in D$, as desired. 

    Thus, it suffices to prove that for every $n<\omega$ and a linear iteration $j\colon L[\vec E]\to M$ by extenders on the $L[\vec{E}]$-sequence whose length is $n$, $j[G]$ generates a $j(\po)$-generic set over $M$. 
    
    We proceed by induction on $n$. Assume that the statement holds for iterations of length $n$, and let $j\colon V\to M$ be an iteration of length $n+1$. Let $i\colon V\to N$ be the initial segment of $j$ of length $n$, and let $F\in\vec E^N$ be such that $M=\Ult(N,F)$ and $j=j_F\circ i$. Suppose that $F=\vec E^N(\alpha,\beta)$. By induction, $i[G]$ generates an $i(\po)$-generic set over $N$. By Lemma \ref{Lem:CanonicalLiftOfAnExtender}, the extender ultrapower embedding $j_F\colon N\to M$ lifts to $j^*_F\colon N[i[G]]\to M[j[G]]$, which is also an extender embedding with respect to the $(\alpha, \alpha+1+\beta)$-extender $F^*$ which is derived from $j^*_F$.

    At this point, the proof given above actually shows that, for every ordinal $\alpha\leq \text{lh}(j)$, letting $j_\alpha\colon L[\vec E]\to M_\alpha$ consist of the first $\alpha$-many steps in the iteration, $j_\alpha[G]$ generates a $j_\alpha(\po)$-generic set over $M_\alpha$. Denote by $H_\alpha$ the $j_{\alpha}(\po)$-generic generated from $j_\alpha[G]$. Note that at each successor step, $H_{\alpha+1}\in M_\alpha[H_\alpha]$, since $H_{\alpha+1}$ coincides with the generic set generated from $j_{\alpha,\alpha+1}[ H_\alpha ]$. It thus follows that the canonical lift $j^*_{\alpha,\alpha+1}$ which satisfies $j^*_{\alpha,\alpha+1}(H_\alpha) = H_{\alpha+1}$, is an internal extender ultrapower of $M_\alpha[H_\alpha]$. It is also straightforward that for every limit ordinal $\alpha$, $M_{\alpha}[H_\alpha]$ is the direct limit of the system $\la \la M_\beta[H_\beta] \colon \beta<\alpha \ra, \la j^*_{\beta,\beta'} \colon \beta\leq \beta'\leq \alpha \ra \ra$. Letting $H\subseteq j(\po)$ be the generic set generated from $j[G]$, it follows that $j^*\colon V[G]\to M[H]$ is a linear, internal iterated extender ultrapower of $L[E][G]$. 
\end{proof}

\begin{remark}
Theorem \ref{Thm:LiftingNormalIteratesAfterDiscreteProcuctForcing} didn't assume definability of the iteration $j\colon L[\vec{E}]\to M$ in $L[\vec{E}]$. We just require that the iteration is internal: that is, at each successor step, the extender ultrapower is taken with respect to an extender that belongs to the model over which it is applied. It is not hard to see that the lift $j^*$ is then also an internal iteration of $L[\vec{E}][G]$. Furthermore, if $j$ is definable in $L[\vec{E}]$ to begin with, then its lift $j^*$ is definable in $L[\vec E][G]$.
    
\end{remark}

Before we prove Theorem \ref{Thm:MainThmForStrong}, let us state the following well-known result, which is most likely due to Mitchell.

\begin{theorem} \label{Thm:RestrictionsAreIterations}
Let $G$ be generic over $V=L[\vec E]$, and suppose that, in $V[G]$,
\[
j\colon V[G]\to M
\]
is an elementary embedding where $V[G]\vDash {}^\omega M \subseteq M$. Then $j\restriction V\colon V\to j(V)$
is a normal linear iteration of $V=L[\vec E]$ by extenders
from the $L[\vec E]$-sequence.
\end{theorem}

Theorem \ref{Thm:RestrictionsAreIterations} belongs to a sequence of results characterizing restrictions of elementary embeddings to canonical inner models as iteration maps. This line of work began with Kunen (see \cite{KunenMeasures}) and was subsequently extended by Mitchell (see \cite{Mitchellcore}) and Schindler (see \cite{Schi}).

We remark that, in the setting of Theorem \ref{Thm:RestrictionsAreIterations}, although $j\uhr V$ is an iteration of $L[\vec E]$, it need not be definable over $L[\vec E]$; see, for example, \cite[Section 5.2]{GitikKaplan}.

\begin{theorem}\label{Thm:EveryEmbeddingIsAFiniteIteration}
    Let $\vec{E}$ be a coherent sequence of nonoverlapping extenders. Let $\po = \prod_{\alpha\in I} \qo_\alpha$ be a suitably spaced discrete product forcing that satisfies the hypotheses of Definition \ref{Def:discrete-product-forcings}. Let $G\subseteq \po$ be generic over $L[\vec{E}]$. Assume that $M$ is an inner model of $L[\vec E][G]$ such that 
    $$({}^\omega M)\cap L[\vec E][G]\subseteq M.$$
    Let $j\colon L[\vec E][G]\to M $ be an elementary embedding definable in $L[\vec{E}][G]$. Then $j$ is a finite iteration of $L[\vec{E}][G]$ by extenders on the $\vec{E}^*$-sequence.
\end{theorem}

\begin{proof}
    By Theorem \ref{Thm:RestrictionsAreIterations}, $M$ has the form $N[H]$ where $H\subseteq j(\po)$ is generic over $N$, and $N$ is the target of a linear, normal iteration  $\pi:=j\uhr L[\vec{E}] \colon L[\vec{E}]\to N $ of $L[\vec E]$ by its extenders. 
    
    We argue that $\pi$ is a finite iteration. Assume otherwise, and let $t = \la \kappa_n \colon n<\omega \ra\in [\mbox{Ord}]^{\omega}$ be the sequence of the first $\omega$-many critical points of extenders participating in $\pi$. Then $t\in ({}^\omega M)\cap L[\vec E][G]$, and thus $t\in M$. Since by elementarity $M = N[ j(G) ]$ and $j(\po)$ is a $\sigma$-closed forcing, we deduce that $t\in N$. However, the sequence of first $\omega$-many critical points in an iteration cannot belong to the final model in that iteration.

    It follows that $\pi\colon L[\vec E]\to N$ is a finite, normal linear iteration. By Theorem \ref{Thm:LiftingNormalIteratesAfterDiscreteProcuctForcing},  $\pi[G] = j[G]$ generates a $\pi(\po) = j(\po)$-generic set  over $N$.  Since $\pi[G] = j[G]\subseteq H$, and $H$ is $j(\po)$-generic over $N$, we deduce that $H$ is the generic set generated from $\pi[G]$, and $j\colon V[G]\to M = N[H]$ is the canonical lift of $\pi$. By Theorem \ref{Thm:LiftingNormalIteratesAfterDiscreteProcuctForcing}, $j$ is an iterated ultrapower by extenders on $L[\vec {E}][G]$, whose length is the same length as of $\pi$.
\end{proof}

We now proceed to the proof of our main result, Theorem \ref{Thm:MainThmForStrong}. 

\begin{proof}[Proof of Theorem \ref{Thm:MainThmForStrong}]
    We argue that $L[\vec{E}][G]\vDash {\rm UA}$. Let $U_0, U_1\in V[G]$ be $\sigma$-complete ultrafilters. By Theorem \ref{Thm:EveryEmbeddingIsAFiniteIteration}, there are finite, normal iterations $\pi_0\colon L[\vec E]\to M_0$, $\pi_1\colon L[\vec E]\to M_1$ such that $j^{L[\vec E][G]}_{U_0}, j^{L[\vec E][G]}_{U_1}$ are the canonical lifts of $\pi_0, \pi_1$, respectively. Moreover, $H_0:=j^{L[\vec E][G]}_{U_0}(G)$ and $H_1:=j^{L[\vec E][G]}_{U_1}(G)$ are generated from $\pi_0[G], \pi_1[G]$, respectively, and
    $$ \Ult(L[\vec E][G], U_0)\simeq M_0[H_0] \quad \text{and} \quad \Ult( L[\vec E][G], U_1 )\simeq M_1[H_1].$$
    
    Working in $L[\vec{E}]$, perform the usual least-disagreement comparison of $M_0, M_1$. It thus follows that there are  iterates $N_0, N_1$ and normal, linear iterations by extenders on the ${\vec E}^{M_0}$, ${\vec{E}}^{M_1}$-sequences, $i_0\colon M_0\to N_0$ and $i_1\colon M_1\to N_1$, respectively, such that either $N_0\unlhd N_1$ or $N_1\unlhd N_0$.\footnote{By $N_0\unlhd N_1$ we mean that the extender sequence $\vec{E}^{N_0}$ of $N_0$ is an initial segment of the extender sequence $\vec{E}^{N_1}$, in the sense of the lexicographic order.}
    
    We argue that $N_0 = N_1$. By the minimality of $L[\vec{E}]$,  
    $$L[\vec{E}]\vDash \text{ no strict initial segment $\vec{E}'$ of }\vec{E} \text{ witnesses that }L[\vec{E}'] \text{ has a strong cardinal}.$$
    This property propagates to $N_0, N_1$ and their extender sequences, by elementarity. Thus, $\vec{E}^{N_0} = \vec{E}^{N_1}$, namely $N_0 = N_1$. Denote $N = N_0 = N_1$. 

    \begin{claim}
        $i_0\circ \pi_0 = i_1\circ \pi_1$.
    \end{claim}

    \begin{proof}
        We prove the statement in more general settings: whenever $\pi_0\colon L[\vec{E}]\to M^0, \pi_1\colon L[\vec{E}]\to M^1$ are finite iterates of $L[\vec{E}]$ (for some inner models $M^0, M^1$), if $i_0\colon M^0\to N$ and $i_1\colon N^1\to N$ are the embeddings obtained by performing the least-disagreement comparison as above, then $i_0\circ \pi_0 = i_1\circ \pi_1$.

        We generalized the statement for all pairs of finite iterations in order to perform the proof by observing the least counterexample. Since $\pi_0, \pi_1$ can be associated with finite sequences of extenders, the least counterexample is first order definable in $L[\vec{E}]$: Namely, assume that $x\in L[\vec{E}]$ is the least with respect to the canonical well-order of $L[\vec{E}]$, for which there are finite iterations $\pi_0\colon L[\vec{E}]\to M^0, \pi_1\colon L[\vec{E}]\to M^1$ such that $(i_0\circ \pi_0)(x)\neq (i_1\circ \pi_1)(x)$, where $i_0, i_1$ are as above. Denote $x^*_0:=(i_0\circ \pi_0)(x)$ and $x^*_1:= (i_1\circ \pi_1)(x)$. Then $x^*_0\neq x^*_1$, but both elements are definable in $N$ using the same first order formula, which is a contradiction.
    \end{proof}

    Overall, we obtained a comparison $i_0\circ \pi_0 = i_1\circ \pi_1 \colon L[\vec{E}]\to N$, where $i_0, i_1$ are normal, linear iterations (definable in $V$). 
    
    Apply Theorem \ref{Thm:LiftingNormalIteratesAfterDiscreteProcuctForcing} in order to lift $i_0\colon M_0\to N$ to $i^*_0\colon M_0[ H_0 ]\to N[ H^*_0 ]$, where $H^*_0$ is the generic generated from $i_0[H_0]$. Similarly, lift $i_1$ to $i^*_1\colon M_1[H_1]\to N[H^*_1]$. Note that by Theorem \ref{Thm:LiftingNormalIteratesAfterDiscreteProcuctForcing}, $i^*_0, i^*_1$ are iterations by extenders, which are internal in the sense of $M_0[H_0], M_1[H_1]$, respectively (but not necessarily definable in $M_0[H_0], M_1[H_1]$, respectively).  Furthermore, $i_0[H_0]$ and $i_1[H_1]$ coincide, because both are generated from conditions of the form
    $$ (i_0\circ \pi_0)(p) = (i_1\circ \pi_1)(p)$$
    for $p\in G$. It thus follows that $H^*_0 = H^*_1$, and, consequently,  $i^*_0\circ \pi^*_0 = i^*_1\circ \pi^*_1$.

    Recall that both $i^*_0, i^*_1$ are iterations of extender ultrapowers. Each extender participating in one of the iterations is internal to the model over which it is applied. In particular, $i^*_0, i^*_1$ are obtained by composing and taking direct limit of close embeddings. Consequently, $i^*_0, i^*_1$ are themselves close embeddings (see \cite[Lemma 2.2.23]{goldbergtheultrapoweraxiom}). Namely $W_0, W_1\in V[G]$ admit a comparison by close embeddings. By Lemma \ref{Lem:ComparisonByCloseEmbeddings}, $W_0, W_1$ also admit a comparison by internal $\sigma$-complete ultrafilters.
\end{proof}

\begin{proof}[Proof of Corollary \ref{Cor:UA+Strong+FarFromHOD}]
    Fix $I\subseteq \kappa$ a suitably spaced discrete set of inaccessible cardinals (for instance, the one from Remark \ref{Remark:ExistenceOfSuitablySpacedSets}). Let $\po = \prod_{\alpha\in I}\qo_\alpha$, where for every $\alpha<\kappa$, $\qo_\alpha = \text{Add}(\alpha,1)$. In $L[\vec{E}][G]$, $\kappa$ is strong (by Corollary \ref{Cor:PreserviungStrongness}) and {\rm UA} holds (by Theorem \ref{Thm:MainThmForStrong}). Finally, arguing exactly as in the proof of Corollary \ref{Cor:UA+Meas+VneqHODX}, 
    $$L[\vec{E}][G] \vDash \forall X\in V_\kappa\ (V\neq {\rm HOD}_X).$$
\end{proof}



We remark that the proof of Theorem \ref{Thm:MainThmForStrong} can be modified to show the following:

\begin{theorem}[Close Comparison Theorem]
    Assume that $V = L[\vec{E}]$ where $\vec{E}$ is a coherent sequence of nonoverlapping extenders. Let $\po = \prod_{\alpha\in I}\qo_\alpha$ be a suitably spaced discrete product forcing. Let $G\subseteq \po$ be generic over $L[\vec E]$. Let $j_0\colon L[\vec{E}][G]\to M^*_0, j_1\colon L[\vec{E}][G]\to M^*_1$ be elementary embeddings definable in $L[\vec{E}][G]$, and assume that both $M^*_0, M^*_1$ are inner models of $L[\vec{E}][G]$ which are closed under countable sequences inside $L[\vec E][G]$. Then there is an inner model $N^*$ of $L[\vec E][G]$ and close embeddings $i_0\colon M^*_0\to N^*$, $i_1\colon M^*_1\to N^*$ such that $i_0\circ j_0 = i_1\circ j_1$. 
\end{theorem}

\begin{proof}
    Mimic the proof of Theorem \ref{Thm:MainThmForStrong}, replacing the ultrapower embeddings of $U_0, U_1$ with the embeddings $j_0, j_1$.
\end{proof}

\bibliography{citations}
\bibliographystyle{plain}

\end{document}